\documentclass{article}

\usepackage{amsmath}
\usepackage{amsthm}
\usepackage{amssymb}
\usepackage{amscd}
\usepackage{txfonts}
\usepackage[dvipdfm]{graphicx}
\numberwithin{equation}{section}

\newtheorem{thm}{Theorem}[section]
\newtheorem{lem}{Lemma}[section]
\newtheorem{pro}{Proposition}[section]

\newtheorem{Def}{Definition}[section]

\begin{document}

\title{Hydrodynamic limit for weakly asymmetric simple exclusion processes in crystal lattices}
\author{Ryokichi Tanaka
\footnote{Department of Mathematics, Kyoto University, Kitashirakawa Oiwake-cho, Sakyo-ku, Kyoto, 606-8502, Japan}
\footnote{E-mail: rtanaka@math.kyoto-u.ac.jp}
\footnote{Present address: Advanced Institute for Materials Research and Mathematical Institute, Tohoku University, 2-1-1 Katahira, Aoba-ku, Sendai, 980-8577, Japan}
\footnote{E-mail: rtanaka@wpi-aimr.tohoku.ac.jp}}
\date{}

\maketitle

\begin{abstract}
We investigate the hydrodynamic limit for weakly asymmetric simple exclusion processes in crystal lattices.
We construct a suitable scaling limit by using a discrete harmonic map.
As we shall observe,
the quasi-linear parabolic equation in the limit is defined on a flat torus and depends on both the local structure of the crystal lattice and the discrete harmonic map.
We formulate the local ergodic theorem on the crystal lattice by introducing the notion of local function bundle, which is a family of local functions on the configuration space.
The ideas and methods are taken from the discrete geometric analysis to these problems.
Results we obtain are extensions of ones by Kipnis, Olla and Varadhan to crystal lattices. 
\end{abstract}

\section{Introduction}

The purpose of this paper is to discuss the hydrodynamic limit for interacting particle systems in the crystal lattice.
Problems of the hydrodynamic limit have been studied intensively in the case where the underlying space is the Euclidean lattice.
We extend problems to the case where the underlying space has geometric structures: the {\it crystal lattice}.
The crystal lattice is a generalization of classical lattice, the square lattice, the triangular lattice, the hexagonal lattice, the Kagom\'{e} lattice (Figure\ref{crystals}) and the diamond lattice. 
Before explaining difficulties for this extension and entering into details, we motivate to study these problems.

\begin{figure}[h]
	\begin{center}
	\includegraphics[width=129mm]{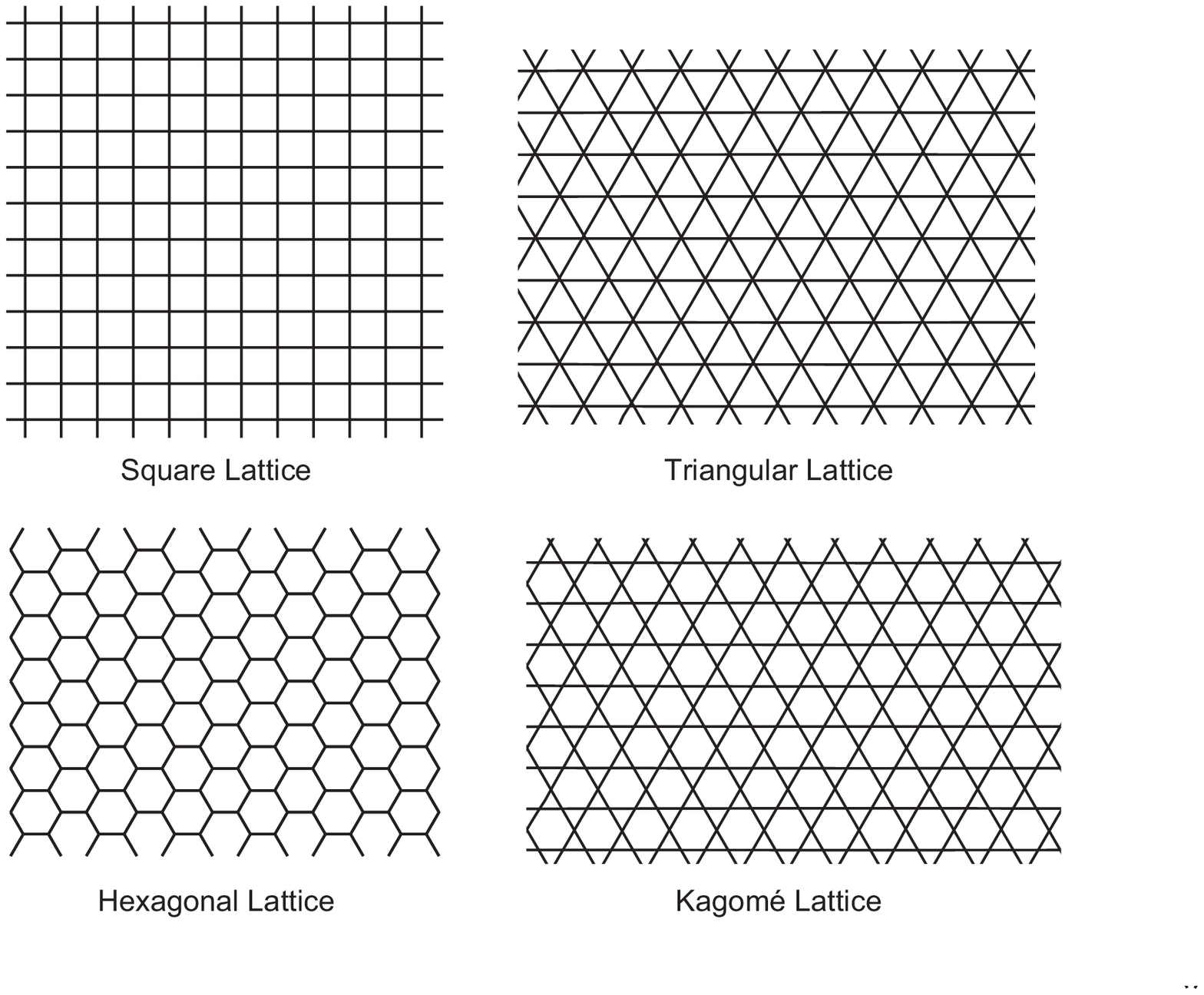}
	\end{center}
	\caption{Crystal Lattices}
	\label{crystals}
\end{figure}

There are many problems on the scaling limit of interacting particle systems, which have their origins in the statistical mechanics and the hydrodynamics.
(See \cite{KL}, \cite{S} and references therein.)
The hydrodynamic limit for the exclusion process is one of the most studied models in this context.
Here we give only one example for exclusion processes in the integer lattice, which is a prototype of our results, due to Kipnis, Olla and Varadhan (\cite{KOV}).
From the view point of physics and mathematics, it is natural to ask for the scaling limit of interacting particle systems evolving in more general spaces and to discuss the relationship between macroscopic behaviors of particles and geometric structures of the underlying spaces.
In this paper, we deal with the crystal lattice, which is the simplest extension of the Euclidean lattice $\mathbb{Z}^{d}$.
Although the crystal lattice has periodic global structures, it has inhomogeneous local structures.

On the other hand, crystal lattices have been studied in view of discrete geometric analysis by Kotani and Sunada (\cite{KS01}, \cite{KS02}, \cite{KS03}, and the expository article \cite{Su}).
They formulate a crystal lattice as an abelian covering graph, and then
they study random walks on crystal lattices and discuss the relationship between asymptotic behaviors of random walks and geometric structures of crystal lattices.
In \cite{KS02}, they introduce the {\it standard realization}, which is a discrete harmonic map from a crystal lattice into a Euclidean space to characterize an equilibrium configuration of crystals.
In \cite{KS01}, they discuss the relationship between the {\it Albanese metric} which is introduced into the Euclidean space, associated with the standard realization and the central limit theorem for random walks on the crystal lattice.
Considering exclusion processes on the crystal lattice,
one is interested to ask what geometric structures appear in the case where the interactions depend on the local structures.

Given a graph, the exclusion process  on it describes the following dynamics:
Particles attempt to jump to nearest neighbor sites, however, they are forbidden to jump to sites which other particles have already occupied.
So, particles are able to jump to nearest neighbor vacant sites.
Then, the problem of the hydrodynamic limit is to capture the collective behavior of particles via the scaling limit.
If we take a suitable scaling limit of space and time,
then we observe that the density of particles is governed by a partial differential equation as a macroscopic model.
Here it is necessary to construct a suitable scaling limit for a graph and to know some analytic properties of the limit space.

A crystal lattice is defined as an infinite graph $X$ which admits a free action of a free abelian group $\Gamma$ with a finite quotient graph $X_{0}$.
We construct a scaling limit of a crystal lattice as follows:
Let $N$ be a positive integer. 
Take a finite index subgroup $N\Gamma$ of $\Gamma$, which is isomorphic to $N\mathbb{Z}^{d}$ when $\Gamma$ is isomorphic to $\mathbb{Z}^{d}$. 
Then we take the quotient of $X$ by $N\Gamma$-action: $X_{N}$.
We call this finite quotient graph $X_{N}$ the {\it $N$-scaling finite graph}.
The quotient group $\Gamma_{N}:=\Gamma/N\Gamma$ acts freely on $X_{N}$.
Here we consider exclusion processes on $X_{N}$.
To observe these processes in the continuous space, we embed $X_{N}$ into a torus.
We construct an embedding map $\Phi_{N}$ from $X_{N}$ into a torus by using a harmonic map $\Phi$ in the discrete sense in order that the image $\Phi_{N}(X_{N})$ converges to a torus as $N$ goes to the infinity.
(Here the convergence of metric spaces is verified by using the Gromov-Hausdorff topology, however, we do not need this notion in this paper.)
Then we obtain exclusion processes embedded by $\Phi_{N}$ into the torus.

In this paper, we deal with the simplest case among exclusion processes: {\it the symmetric simple exclusion process} and its perturbation: {\it the weakly asymmetric simple exclusion process}.
In the latter case, we obtain a heat equation with nonlinear drift terms on torus as the limit of process of empirical density  (Theorem\ref{main} and Examples below).
We observe that the diffusion coefficient matrices and nonlinear drift terms can be computed by data of a finite quotient graph $X_{0}$ and a harmonic map $\Phi$. (See also examples in Section \ref{harmonic}.)
The hydrodynamic limit for these processes on the crystal lattice is obtained as an extension of the one on $\mathbb{Z}^{d}$.
So, first, we review the outline of the proof for $\mathbb{Z}^{d}$, following the method by Guo, Papanicolaou and Varadhan in \cite{GPV}.
Since the lattice $\mathbb{Z}^{d}$ is naturally embedded into $\mathbb{R}^{d}$, 
the combinatorial Laplacian on the scaled discrete torus converges to the Laplacian on the torus according to this natural embedding.
The local ergodic theorem is the key step of the proof since it enables us to replace local averages by global averages and to verify the derivation of the limit partial differential equation.
It is formulated by using local functions on the configuration space and the shift action on the discrete torus.
The proof of the local ergodic theorem is based on the one-block estimate and the two-blocks estimate.
Roughly speaking, the one-block estimate is interpreted as the local law of large numbers and the two-blocks estimate is interpreted as the asymptotic independence of two different local laws of large numbers.

Second, we look over the outline of the proof for the crystal lattice.
There are two main points with regard to the difference between $\mathbb{Z}^{d}$ and the crystal lattice, that are the convergence of the Laplacian and the local ergodic theorem.
Although the crystal lattice $X$ is embedded into an Euclidean space by a harmonic map $\Phi$,
the combinatorial Laplacian on the image of the $N$-scaling finite graph $\Phi_{N}(X_{N})$ does not converge to the Laplacian on the torus straightforwardly.
It is proved by averaging each fundamental domain by $\Gamma$-action because of the local inhomogeneity of the crystal lattice.
Thus, it is necessary to obtain the local ergodic theorem compatible with the convergence of the Laplacian.
Furthermore, it is also necessary to obtain the local ergodic theorem compatible with the local inhomogeneity of the crystal lattice.
For these reasons, we have to modify the local ergodic theorem in the case of crystal lattices.
To formulate the local ergodic theorem in the crystal lattice,
we introduce the notion of {\it $\Gamma$-periodic local function bundles}.
A $\Gamma$-periodic local function bundle is a family of local functions on the configuration space which is parametrized by vertices periodically.
Moreover,
we introduce two different ways to take local averages of a $\Gamma$-periodic local function bundle.
The first one is to take averages per each fundamental domain as a unit.
The second one is to take averages on each $\Gamma$-orbit.
The local ergodic theorem in the crystal lattice is formulated by using $\Gamma$-periodic local function bundles, two types of local averages and the $\Gamma_{N}$-action on the $N$-scaling finite graph $X_{N}$.
In fact, 
we use only special $\Gamma$-periodic local function bundles to handle the weakly asymmetric simple exclusion process.
The proof of this local ergodic theorem is also based on the one-block estimate and the two-blocks estimate.
Proofs of these two estimates are analogous to the case of the discrete torus since we use the fact that the whole crystal lattice is covered by the $\Gamma$-action of a fundamental domain in the first type of the local average and we restrict to a $\Gamma$-orbit in the second type of the local average.
In this paper, 
we call the local ergodic theorem the {\it replacement theorem} and prove it in the form of the super exponential estimate.
The derivation of the hydrodynamic equation is the same manner as the case of the discrete torus.

Let us mention related works.
Interacting particle systems are categorized into the gradient system and the non-gradient system, according to types of interactions.
We call the system the gradient system when the interaction term is represented by the difference of local functions.
Otherwise,
we call the system the non-gradient system.
We mention a recent work on the non-gradient system by Sasada \cite{Sa}.
The symmetric simple exclusion process is a model of the gradient system.
Our problems essentially correspond to problems for the gradient system since the hydrodynamic limit for the weakly asymmetric simple exclusion process is reduced to the one for the symmetric simple exclusion process, following \cite{KOV}.
As for the hydrodynamic limit on spaces other than the Euclidean lattice, 
Jara investigates the hydrodynamic limit for zero-range processes in the Sierpinski gasket (\cite{J}).
As for the crystal lattice, there is another type of the scaling limit.
In \cite{SS}, Shubin and Sunada study lattice vibrations of crystal lattices and calculate one of the thermodynamic quantities: the specific heat. 
They derive the equation of motion by taking the continuum limit of the crystal lattice.
As a further problem, we mention the following problem:
Recently, attentions have been payed for interacting particle systems evolving in random environments (e.g., \cite{BB}, \cite{F} and \cite{K}).
For example, the quenched invariance principle for the random walk on the infinite cluster of supercritical percolation of $\mathbb{Z}^{d}$ with $d \ge 2$ is proved by Berger and Biskup (\cite{BB}).
Their argument is based on a harmonic embedding of percolation cluster into $\mathbb{R}^{d}$.
Our use of the harmonic map $\Phi$ and local function bundles will play a role in the systematic treatment of particle systems in more general random graphs.
Furthermore, the hydrodynamic limit on the inhomogeneous crystal lattice is considered as the case where the crystal lattice has topological defects.
This problem would be interesting in connection with material sciences.

This paper is organized as follows:
In Section \ref{crystal lattices},
we introduce the crystal lattice and construct the scaling limit by using discrete harmonic maps.
In Section \ref{hydrodynamic limit},
we formulate the weakly asymmetric simple exclusion process on the crystal lattice and state the main theorem (Theorem \ref{main}).
In Section \ref{replacement theorem},
we introduce $\Gamma$-periodic local function bundles and
show the replacement theorem (Theorem \ref{super exponential estimate}).
We prove the one-block estimate and the two-blocks estimate.
In Section \ref{the proof of the main theorem},
we derive the quasi-linear parabolic equation, applying the replacement theorem and complete the proof of Theorem \ref{main}.
Section \ref{appendixA} is Appendix;A. 
We prove some lemmas related to approximation by combinatorial metrics to complete the scaling limit argument.
Section \ref{appendixB} is Appendix;B.
We refer an energy estimate of a weak solution and a uniqueness result for the partial differential equation to this appendix.

{\it Landau asymptotic notation.}
Throughout the paper, we use the notation $a=o_{N}$ to mean that $a \to 0$ as $N \to \infty$.
We also use the notation $a=o_{\epsilon}$ to mean that $a \to 0$ as $\epsilon \to 0$.

\section{The crystal lattice and the harmonic realization}\label{crystal lattices}

In this section, we introduce the crystal lattice as an infinite graph and its realization into the Euclidean space.

\subsection{Crystal lattices}

Let $X=(V,E)$ be a locally finite connected graph, where $V$ is a set of vertices and $E$ a set of all oriented edges. The graph $X$ may have loops and multiple edges.
For an oriented edge $e \in E$, we denote by $oe$ the origin of $e$, by $te$ the terminus and by $\overline e$ the inverse edge of $e$.
Here we regard $X$ as a weighted graph, whose weight functions on $V$ and $E$ are all equal to one.

We call a locally finite connected graph $X=(V,E)$ a {\it $\Gamma$-crystal lattice} if a free abelian group $\Gamma$ acts freely on $X$ and the quotient graph $\Gamma \backslash X$ is a finite graph $X_{0}=(V_{0},E_{0})$. 
More precisely, each $\sigma \in \Gamma$ defines a graph isomorphism $\sigma : X \to X$ and the graph isomorphism is fixed point-free except for $\sigma=id$.
In other words, a $\Gamma$-crystal lattice $X$ is an abelian covering graph of a finite graph $X_{0}$ whose covering transformation group is $\Gamma$.

\subsection{Harmonic maps}\label{harmonic}

Let us construct an embedding of a $\Gamma$-crystal lattice $X$ into the Euclidean space $\mathbb{R}^{d}$ of dimension $d=rank \Gamma$.

Given an injective homomorphism $\phi: \Gamma \to \mathbb{R}^{d}$ such that
there exits a basis $u_{1}, \dots, u_{d} \in \mathbb{R}^{d}$, 
$$
\phi(\Gamma)=\left\{\sum_{i=1}^{d}k_{i}u_{i} \ | \ \text{$k_{i}$ integers}\right\},
$$

then we define a harmonic map associated with $\phi$.

\begin{Def}
Fix an injective homomorphism $\phi$ as above. We call an embedding $\Phi:X \to \mathbb{R}^{d}$, a {\it $\phi$-periodic harmonic map} if $\Phi$ satisfies the followings:
{\it $\Phi$ is $\Gamma$-periodic}, i.e., for any $x \in V$ and any $\sigma \in \Gamma$, $\Phi(\sigma x)=\Phi(x) + \phi(\sigma)$ and {\it $\Phi$ is harmonic}, i.e., for any  $x \in V$, $\sum_{e \in E_{x}}[\Phi(te)-\Phi(oe)]=0$, where $E_{x}=\{e \in E \ | \ oe=x \}$.
\end{Def}

We note that a $\phi$-periodic harmonic map $\Phi$ depends on $\phi$ and call $\Phi$ a {\it periodic harmonic map} in short when we fix some $\phi$.

For $e \in E_{0}$, we take a lift $\tilde e \in E$ of $e$, and define ${\bf v}(e):=\Phi(t \tilde e)-\Phi(o \tilde e) \in \mathbb{R}^{d}$.
By the $\Gamma$-periodicity, ${\bf v}(e)$ does not depend on the choices of lifts.
For ${\bf v}(e)=(v_{1}(e),\dots,v_{d}(e)) \in \mathbb{R}^{d}$, let us define a $d \times d$-matrix by
$$
\mathbb{D}:=\frac{1}{4|V_{0}|}\left(\sum_{e \in E_{0}}v_{i}(e)v_{j}(e)\right)_{i,j=1,\dots,d}.
$$
Here the matrix is symmetric and positive definite.
We call the matrix $\mathbb{D}$ the {\it diffusion coefficient matrix}.

{\it Examples}

\begin{enumerate}

\item[0.]
The one dimensional standard lattice.

\item[0a.]

The one dimensional standard lattice $X$ which we identify the set of vertices $V$ with $\mathbb{Z}$ and the set of (unoriented) edges with the set of pairs of vertices
$\{(x, x+1) \ | \ x \in \mathbb{Z}\}$.
Now $\mathbb{Z}$ acts freely on $X$ by the additive operation in $\mathbb{Z}$ and the quotient finite graph consists of one vertex and one loop as un oriented graph.
When we regard $X$ as an oriented graph,
we add both oriented edges to $X$ and the quotient graph consists of one vertex and two oriented loops (Figure \ref{one-dim}).

Let us choose a canonical injective homomorphism $\phi: \mathbb{Z} \to \mathbb{R}$.
In our formulation, 
choose a basis $e_{1}=1$ in $\mathbb{R}$ and define 
$\phi: \mathbb{Z} \to \mathbb{R}$ by setting $\phi(n)=ne_{1}$ for $n \in \mathbb{Z}$ 
so that $\phi(\mathbb{Z})=\{ne_{1} \ | \ n \in \mathbb{Z}\}$.
By identifying the set of vertices of $X$ with $\mathbb{Z}$,
we define an embedding map $\Phi(x)=\phi(x)$, $x \in \mathbb{Z}$.
This embedding map $\Phi$ is a $\mathbb{Z}$-periodic harmonic map.
In this case,
$\mathbb{D}=1/2$.

\begin{figure}[h]
	\begin{center}
	\includegraphics[width=129mm]{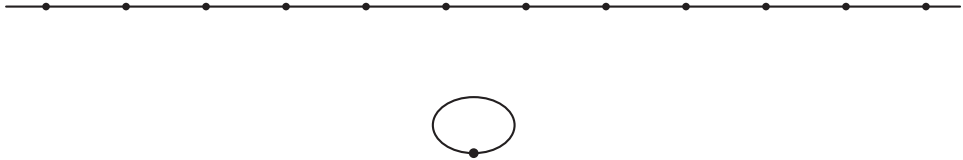}
	\end{center}
	\caption{The one dimensional standard lattice and the quotient graph in Example 0a.}
	\label{one-dim}
\end{figure}

\item[0b.]

Let us give another example of periodic harmonic map for the one dimensional standard lattice $X$.
Now we define a $\mathbb{Z}$-action on $X$ in the following way:
For $\sigma \in \mathbb{Z}$, $x \in V$, define $\sigma x:=2\sigma + x$.
Then this induces a free $\mathbb{Z}$-action on $X$ and the quotient graph consists of two vertices and two edges between them as an unoriented graph.
Let $\phi$ be the injective homomorphism as the same as in Example 0a.
We define an embedding map $\Phi : X \to \mathbb{R}$ by setting 
$\Phi(\sigma 0):=0 + \phi(\sigma)$, $\Phi(\sigma 1):= -1 +\phi(\sigma)$.
Then $\Phi$ is a periodic harmonic map.
The image of $\Phi$ is not isomorphic to the previous one (Figure \ref{one-dim2}).
In this case, $\mathbb{D}=5/4$.

\begin{figure}[h]
	\begin{center}
	\includegraphics[width=129mm]{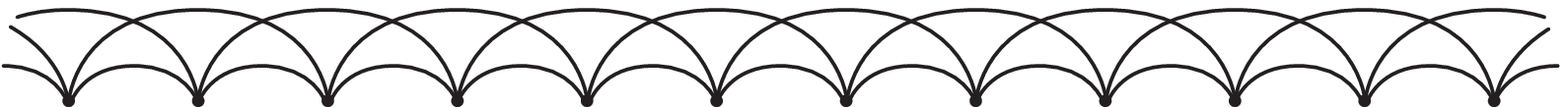}
	\end{center}
	\caption{The image of  $\Phi$ in Example 0b.}
	\label{one-dim2}
\end{figure}

\item[1.] The square lattice.

\item[1a.]
The square lattice has the standard embedding in $\mathbb{R}^{2}$ and this embedding is shown to be periodic and harmonic in our sense in the following.
We identify the set of vertices of the square lattice $X$ with $\mathbb{Z}^{2}$ and the set of edges with the set of pairs of vertices $\{(x, x + (1, 0)), (x, x + (0, 1)) \ | \ x \in \mathbb{Z}^{2})\}$.
Now $\mathbb{Z}^{2}$ acts freely on $X$ by the additive operation in $\mathbb{Z}^{2}$
and the quotient graph is the bouquet graph with one vertex and two unoriented loops.
When we regard $X$ as an oriented graph, we add both oriented edges to $X$ and the quotient finite graph is the bouquet graph with one vertex and four oriented loops.

Let us choose a canonical injective homomorphism $\phi: \mathbb{Z}^{2} \to \mathbb{R}^{2}$.
That is, choose a basis $\{e_{1}=(1, 0), e_{2}=(0, 1)\}$ in $\mathbb{R}^{2}$ and define $\phi: \mathbb{Z}^{2} \to \mathbb{R}^{2}$ by setting $\phi((m, n))=m e_{1} + n e_{2}$ for $(m, n) \in \mathbb{Z}^{2}$ so that
$\phi(\mathbb{Z}^{2})=\{\sum_{i=1}^{2}k_{i}e_{i} \ | \ k_{i} \ \text{integers}\}$.
By identifying the set of vertices of $X$ with $\mathbb{Z}^{2}$, we define an embedding map $\Phi(x)=\phi(x)$, $x \in \mathbb{Z}^{2}$.
This embedding map $\Phi$ is a $\mathbb{Z}^{2}$-periodic harmonic map.
In this case,
$\mathbb{D}=
\begin{pmatrix}
1/2	& 0 \\
0	& 1/2
\end{pmatrix}.
$

\item[1b.]
Let us give another example of periodic harmonic map for the square lattice $X$.
Choose a basis $\{u_{1}=(1, 0), u_{2}=(1/2, 1)\}$ in $\mathbb{R}^{2}$ and define $\phi: \mathbb{Z}^{2} \to \mathbb{R}^{2}$ by setting $\phi((m, n))=m u_{1} + n u_{2}$ for $(m, n) \in \mathbb{Z}^{2}$ so that
$\phi(\mathbb{Z}^{2})=\{\sum_{i=1}^{2}k_{i}u_{i} \ | \ k_{i} \ \text{integers}\}$.
In the same way as above Example 1a, we define an embedding map $\Phi(x)=\phi(x)$, $x \in \mathbb{Z}^{2}$.
 (Figure \ref{square}.)
This embedding map $\Phi$ is a $\mathbb{Z}^{2}$-periodic harmonic map.
In this case, 
$\mathbb{D}=
\begin{pmatrix}
5/8	& 1/4 \\
1/4	& 1/2
\end{pmatrix}.
$

\begin{figure}[h]
	\begin{center}
	\includegraphics[width=129mm]{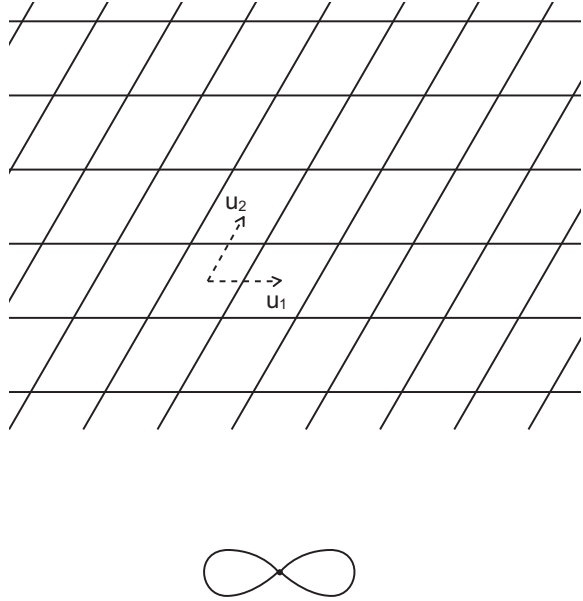}
	\end{center}
	\caption{The square lattice embedded as in Example 1b. and the quotient graph.}
	\label{square}
\end{figure}

\item[1c.]
Let us give an example an embedding map $\Phi$ which is periodic but not harmonic.
We choose an action of $\mathbb{Z}^{2}$ on the square lattice $X$ in the following way:
Again, we identify the set of vertices $V$ of $X$ with $\mathbb{Z}^{2}$.
For $\sigma = (\sigma_{1}, \sigma_{2}) \in \mathbb{Z}^{2}$, $x=(x_{1}, x_{2}) \in V$, define $\sigma x:=(\sigma_{1}+x_{1}, 2\sigma_{2} + x_{2})$.
Then this induces a free $\mathbb{Z}$-action and the quotient graph consists of two vertices, two edges between them and one loop on each vertex (two loops) as an unoriented graph.
Let $\phi$ be the same as in Example 1a. 
We define an embedding map $\Phi': X \to \mathbb{R}^{2}$ by setting
$\Phi'(\sigma (0,0))=(0,0) + \phi(\sigma)$, $\Phi'(\sigma (0,1)) =(1, 1/2) + \phi(\sigma)$ for $\sigma \in \mathbb{Z}^{2}$.
Then $\Phi'$ is periodic but not harmonic since for $x=(0,0) \in \mathbb{Z}^{2}$,
$\sum_{e \in E_{x}}\left[\Phi'(te)- \Phi'(oe)\right] = (1, 0)+(-1, 0) +(1, 1/2) + (1, -1/2)=(2, 0) \neq (0,0)$.

\item[2.]
The hexagonal lattice.

The hexagonal lattice admits a free $\mathbb{Z}^{2}$-action with the quotient graph consisting of two vertices and three edges as an unoriented graph.
We define a fundamental subgraph $D$ by setting the set of vertices $\{x_{0}, x_{1}, x_{2}, x_{3}\}$ and the set of (unoriented) edges $\{(x_{0}, x_{1}), (x_{0}, x_{2}), (x_{0}, x_{3})\}$.
Then the hexagonal lattice has a subgraph isomorphic to $D$ and is covered by copies of the subgraph translated by the $\mathbb{Z}^{2}$-action. 
Choose a basis $\{u_{1}=(\sqrt{3}, 0), u_{2}=(\sqrt{3}/2, 3/2)\}$ in $\mathbb{R}^{2}$ and define $\phi: \mathbb{Z}^{2} \to \mathbb{R}^{2}$ by setting $\phi((m, n))=m u_{1} + n u_{2}$ for $(m, n) \in \mathbb{Z}^{2}$ so that
$\phi(\mathbb{Z}^{2})=\{\sum_{i=1}^{2}k_{i}u_{i} \ | \ k_{i} \ \text{integers}\}$.
We define an embedding map $\Phi$ by setting $\Phi(\sigma x_{0})=(0,0) + \phi(\sigma)$, $\Phi(\sigma x_{1})=(-\sqrt{3}/2, 1/2) + \phi(\sigma)$, $\Phi(\sigma x_{2})=(\sqrt{3}/2, 1/2) +\phi(\sigma)$ and $\Phi(\sigma x_{3})=(0, -1) + \phi(\sigma)$ for $\sigma \in \mathbb{Z}^{2}$.
(Figure \ref{hexagonal}.)
Then $\Phi$ is a periodic harmonic map.
In this case,  
$\mathbb{D}=
\begin{pmatrix}
3/8	& 0 \\
0	& 3/8
\end{pmatrix}.
$

\begin{figure}[h]
	\begin{center}
	\includegraphics[width=100mm]{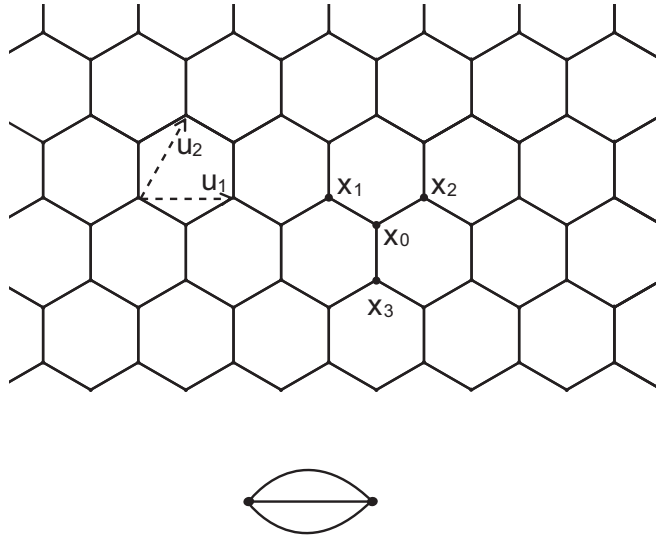}
	\end{center}
	\caption{The hexagonal lattice embedded as in Example 2. and the quotient graph.}
	\label{hexagonal}
\end{figure}

\item[3.]
The Kagom\'{e} lattice.

The Kagom\'{e} lattice admits a free $\mathbb{Z}^{2}$-action with the quotient graph consisting of three vertices and six edges (two edges between each pair of vertices) as an unoriented graph.
We define a fundamental subgraph $D$ by setting the set of vertices $\{x_{0}, x_{1}, x_{2}, x_{3}, x_{4}\}$ and the set of (unoriented) edges $\{(x_{0}, x_{1}), (x_{0}, x_{2}), (x_{0}, x_{3}), (x_{0}, x_{4})\}$.
Then the Kagom\'{e} lattice has a subgraph isomorphic to $D$ and is covered by copies of the subgraph translated by the $\mathbb{Z}^{2}$-action.
Choose a basis $\{u_{1}=(\sqrt{3}, 0), u_{2}=(\sqrt{3}/2, 3/2)\}$ in $\mathbb{R}^{2}$
and define $\phi: \mathbb{Z}^{2} \to \mathbb{R}^{2}$ as the same as in Example 2.
We define an embedding map $\Phi$ by setting
$\Phi(\sigma x_{0})=(0,0)+\phi(\sigma)$,
$\Phi(\sigma x_{1})=(-\sqrt{3}/2,0)+\phi(\sigma)$,
$\Phi(\sigma x_{2})=(-\sqrt{3}/4,-3/4)+\phi(\sigma)$,
$\Phi(\sigma x_{3})=(\sqrt{3}/2, 0)+\phi(\sigma)$,
$\Phi(\sigma x_{4})=(\sqrt{3}/4, 3/4)+\phi(\sigma)$
for $\sigma \in \mathbb{Z}^{2}$.
 (Figure \ref{kagome}.)
Then $\Phi$ is a periodic harmonic map.
In this case,
$\mathbb{D}=
\begin{pmatrix}
3/8	& 0 \\
0	& 3/8
\end{pmatrix}.
$

\begin{figure}[h]
	\begin{center}
	\includegraphics[width=100mm]{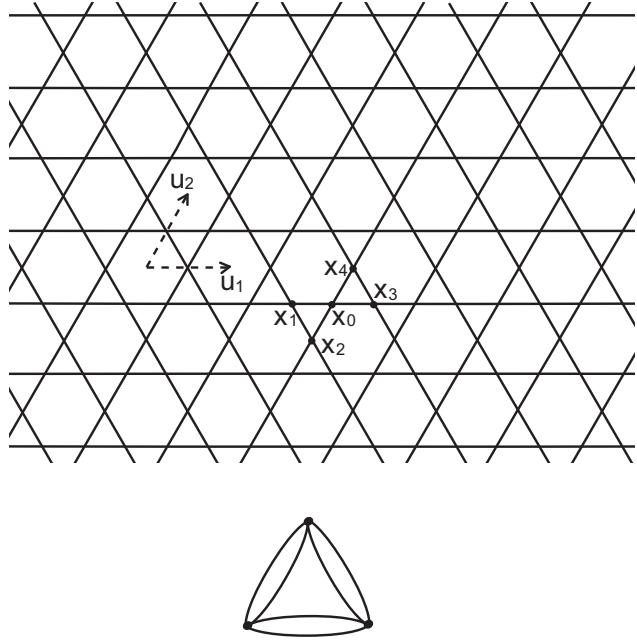}
	\end{center}
	\caption{The Kagom\'{e} lattice embedded as in Example 3. and the quotient finite graph.}
	\label{kagome}
\end{figure}

\end{enumerate}

{\it Remark}.

The notion of periodic harmonic map on $\Gamma$-crystal lattice is studied by Kotani and Sunada and including the standard realization which they introduced in \cite{KS02} as a special case.
They use harmonic maps to characterize {\it equilibrium configurations} of crystals.
In fact, a periodic harmonic map is characterized by a critical map for some discrete analogue of energy functional.
The standard realization is not only a critical map but also the map whose energy itself is minimized by changing flat metrics on torus with fixed volume. 
(More precisely, see\cite{KS02}).
The existence of periodic harmonic map for every injective homomorphism producing lattices in $\mathbb{R}^{d}$ and the uniqueness up to translation is proved in Theorem 2.3 and Theorem 2.4 in \cite{KS02}.

\subsection{Scaling Limits}\label{scaling limit}

Let us construct the scaling limit of the crystal lattice.
Suppose that $\Gamma$ is isomorphic to $\mathbb{Z}^{d}$.
Let $N \ge 1$ be an arbitrary positive integer and $N\Gamma$ the subgroup isomorphic to $N\mathbb{Z}^{d}$. 
The subgroup $N\Gamma$ acts also freely on $X$ and its quotient graph $N\Gamma \backslash X$ is also a finite graph $X_{N}=(V_{N},E_{N})$. 
Then $\Gamma \slash N\Gamma \cong \mathbb{Z}^{d}\slash N\mathbb{Z}^{d}$ acts freely on $X_{N}$. 
We call $X_{N}$ the {\it $N$-scaling finite graph}.
The map 
$$
\frac{1}{N}\Phi:X \to \mathbb{R}^{d},
$$
satisfies that $(1/N)\Phi(\sigma^{N} x )=(1/N)\Phi(x)+\psi(\sigma)$ for all $x \in V$ and all $\sigma \in \Gamma$ since $\Phi$ is $\Gamma$-equivariant. 
We have the torus $\mathbb{T}^{d}:=\mathbb{R}^{d}\slash \psi(\Gamma)$, equipped with the flat metric induced from the Euclidean metric.
The torus depends on $\psi$, however, we do not specify it in the following.
Then the map $(1/N)\Phi:X \to \mathbb{R}^{d}$ induces the map 
$$
\Phi_{N}:X_{N} \to \mathbb{T}^{d}.
$$
We call $\Phi_{N}$ the {\it $N$-scaling map}. (Figure\ref{scaling}.)

$$
\begin{CD}
X @> \frac{1}{N}\Phi >> 	\mathbb{R}^{d} \\
@VVV		@VVV	\\
X_{N}	@>> \Phi_{N} > \mathbb{T}^{d}
\end{CD}
$$

\begin{figure}[h]
	\begin{center}
	\includegraphics[width=129mm]{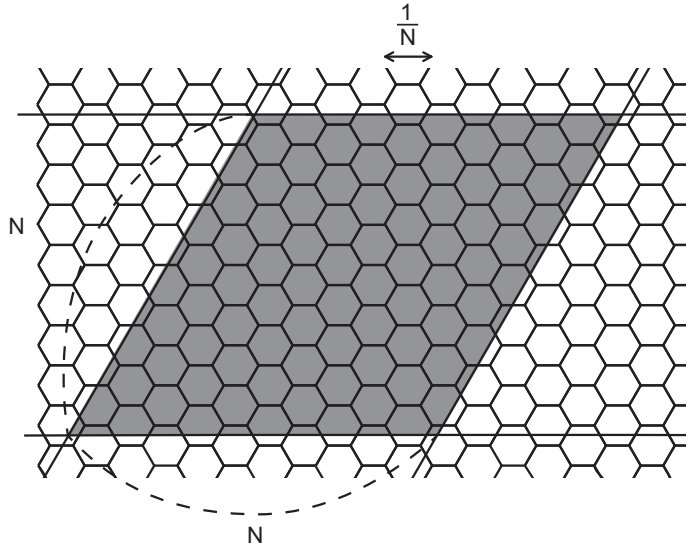}
	\end{center}
	\caption{The image of the $N$-scaling finite graph by a harmonic map in the covering space}
	\label{scaling}
\end{figure}

Next, we observe convergence of the combinatorial Laplacian on $X_{N}$.
Since the degrees of $x \in V_{N}$ might be different, depending on each $x$,
we consider ``average" of the combinatorial Laplacian on a fundamental domain.

Let $d(x)$ be the degree of a vertex $x \in V_{N}$, i.e., the cardinality of the set $E_{N, x}=\{e \in E_{N} \ | \ oe=x\}$. Define the combinatorial Laplacian $\Delta^{c}_{N}$ associated with $X_{N}=(V_{N}, E_{N})$ acting on the space of continuous functions $C(\mathbb{T}^{d})$ by
$$
\Delta^{c}_{N}J(\Phi_{N}(x)):=\frac{1}{d(x)}\sum_{e \in E_{N, x}}\left[J(\Phi_{N}(te)) - J(\Phi_{N}(oe))\right],
$$
for $J \in C(\mathbb{T}^{d})$ and $x \in V_{N}$.
We show that the combinatorial Laplacian converges to the Laplacian on $\mathbb{T}^{d}$ in the following sense:
For every twice continuous derivative functions $J \in C^{2}(\mathbb{T}^{d})$, for every ${\bf x} \in \mathbb{T}^{d}$, for each $x \in V_{0}$, take an arbitrary sequence of vertices $x_{N} \in V_{N}$ such that $x_{N}$ is a lift of $x$ and $\Phi_{N}(x_{N}) \to {\bf x}$ as $N \to \infty$, then by the Taylor formula,
\begin{align*}
&N^{2}\Delta^{c}_{N}J(\Phi_{N}(x_{N})) 	\\
&= \frac{1}{d(x)}\sum_{e \in E_{N,x_{N}}}\left(N\sum_{i=1}^{d}\frac{\partial J}{\partial x_{i}}(\Phi_{N}(x_{N}))v_{i}(e)+\frac{1}{2}\sum_{i,j=1}^{d}\frac{\partial^{2}J}{\partial x_{i}\partial x_{j}}(\Phi_{N}(x_{N}))v_{i}(e)v_{j}(e)\right) + o_{N}.
\end{align*}
Since $\Phi$ is harmonic, 
$$
\frac{1}{|V_{0}|}\sum_{x \in V_{0}}d(x)N^{2}\Delta^{c}_{N}J\left(\Phi_{N}(x_{N})\right) = \frac{1}{2|V_{0}|}\sum_{x \in V_{0}}\sum_{e \in E_{N,x_{N}}}\sum_{i,j=1}^{d}\frac{\partial^{2}J}{\partial x_{i}\partial x_{j}}({\bf x})v_{i}(e)v_{j}(e)  + o_{N}.
$$
Since $\sum_{x \in V_{0}}\sum_{e \in E_N,x_{N}}v_{i}(e)v_{j}(e)=\sum_{e \in E_{0}}v_{i}(e)v_{j}(e)$,
the last term is equal to $2\nabla\mathbb{D}\nabla J({\bf x})$, where $\mathbb{D}$ is a diffusion coefficient matrix and $\nabla \mathbb{D} \nabla=\sum_{i,j}d_{ij}({\partial}^{2}/\partial x_{i}\partial x_{j})$ and $\mathbb{D}=(d_{ij})_{i,j=1,\dots,d}$.

\section{Hydrodynamic limit for exclusion processes}\label{hydrodynamic limit}

We formulate the symmetric simple exclusion process and the weakly asymmetric simple exclusion process in crystal lattices.
As we see below, the former is a particular case of the latter.

Let $X_{N}=(V_{N},E_{N})$ be the $N$-scaling finite graph of $X$.
We denote the configuration space by $Z_{N}:=\{0,1\}^{V_{N}}$.
We denote the configuration space for the whole crystal lattice $X=(V, E)$ by $Z:=\{0, 1\}^{V}$. Each configuration is defined by $\eta=(\eta_{x})_{x \in V_{N}} \in Z_{N}$ with $\eta_{x}=0$ or $1$ and by $\eta \in Z$ in the same way.

We consider the Bernoulli measure $\nu_{\rho}^{N}$ and $\nu_{\rho}$ on $Z_{N}$, $Z$, respectively, for $0 \le \rho \le 1$.
They are defined as the product measures of the Bernoulli measure $\nu_{\rho}^{1}$ on $\{0,1\}$, where $\nu_{\rho}^{1}(0)=1-\rho, \nu_{\rho}^{1}(1)=\rho$.

Let $L^{2}(Z_{N}, \nu_{\rho}^{N})$ be the $L^{2}$-space of $\mathbb{R}$-valued functions on $Z_{N}$. The action of $\Gamma_{N}$ on $X_{N}$ lifts on $Z_{N}$ by setting $(\sigma \eta)_{x}:=\eta_{\sigma^{-1}x}$ for $\sigma \in \Gamma_{N}$ and $x \in V_{N}$.
The group $\Gamma_{N}$ also acts on $L^{2}(Z_{N}, \nu_{\rho}^{N})$ by $\sigma F(\eta):=F(\sigma^{-1}\eta)$ for $F \in L^{2}(Z_{N}, \nu_{\rho}^{N})$.
For $e \in E_{N}$ and $\eta \in Z_{N}$, we denote by $\eta^{e}$ the configuration defined by exchanging the values of $\eta_{oe}$ and $\eta_{te}$, i.e.,
$$
\eta^{e}_{x}:=
\begin{cases}
\eta_{te} & x=oe \\
\eta_{oe} & x=te \\
\eta_{x} & \text{otherwise}.
\end{cases}
$$
For each $e \in E_{N}$, we define the operator $\pi_{e}:L^{2}(Z_{N},\nu_{\rho}^{N}) \to L^{2}(Z_{N},\nu_{\rho}^{N})$ by setting $\pi_{e}F(\eta):=F(\eta^{e})-F(\eta)$.
We see that $\eta^{e}=\eta^{\overline e}$ and $\pi_{e}F=\pi_{\overline e}F$ for $\overline e \in E_{N}$.
The above notations also indicate corresponding ones for $Z=\{0,1\}^{V}$ the configuration space on the whole crystal lattice.

The symmetric simple exclusion process is defined by the generator $L_{N}$ acting on $L^{2}(Z_{N},\nu_{\rho}^{N})$ as 
$$
L_{N}F(\eta)=\frac{1}{4}\sum_{e \in E_{N}} \pi_{e}F(\eta), \ \ \ \ \ F \in L^{2}(Z_{N},\nu_{\rho}^{N}).
$$

The weakly asymmetric simple exclusion process is defined as a perturbation of the symmetric simple exclusion process.
We denote by $C^{1, 2}([0,T]\times \mathbb{T}^{d})$ the space of continuous functions with continuous derivatives in $[0, T]$ and the twice continuous derivatives in $\mathbb{T}^{d}$.
For each function $H \in C^{1, 2}([0,T]\times \mathbb{T}^{d})$, the weakly asymmetric simple exclusion process on $X_{N}$ is defined by the generator $L_{N}^{H}$ acting on $L^{2}(Z_{N},\nu_{\rho}^{N})$ as 
$$
L_{N}^{H}F(\eta)=\frac{1}{2}\sum_{e \in E_{N}} c^{H}(e, \eta, t)\pi_{e}F(\eta), \ \ \ \ \ F \in L^{2}(Z_{N},\nu_{\rho}^{N}),
$$
where
$$
c^{H}(e, \eta, t):=\eta_{oe}\left(1-\eta_{te}\right)\exp\left[H(t, \Phi_{N}(te))- H(t, \Phi_{N}(oe))\right].
$$
Here $\Phi_{N}: X_{N} \to \mathbb{T}^{d}$ is the $N$-scaling map.
The meaning of the perturbation is as follows:
We introduce a ``small'' drift depending on space and time in particles.
In the original process,
a particle jumps with rate $1/2$ from $oe$ to $te$ ($e$ is an edge) at time $t$,
while in the perturbed process, a particle jumps approximately with rate
$$
\frac{1}{2}\left(1+\frac{1}{N}H(t, \Phi_{N}(oe))\right).
$$
Therefore, the external field which is now $(1/2N)\nabla H$ gives a small asymmetry of the order $1/N$ in the jump rate.
Notice that we obtain the symmetric simple exclusion process when $H$ is constant.

{\it Remark}

The weakly asymmetric simple exclusion process which we introduced here does not include the well-studied case where for one dimensional lattice, the external field is $(1/2N)E$ for some constant $E >0$ and its limit equation produces the viscous Burgers equation (e.g., \cite{BLM}).
This process corresponds to the case with $\nabla H \equiv E$ which we do not treat here.

Let $D([0,T], Z_{N})$ be the space of paths which are right continuous and have left limits for some arbitrary fixed time $T>0$.
For a probability measure $\mu^{N}$ on $Z_{N}$, we denote by $\mathbb{P}_{N}^{H}$ the distribution on $D([0,T], Z_{N})$ of the continuous time Markov chain $\eta^{N}(t)$ generated by $N^{2}L_{N}^{H}$ with the initial measure $\mu^{N}$.

The main theorem is stated as follows:

\begin{thm}\label{main}

Let $\rho_{0}:\mathbb{T}^{d} \to [0,1]$ be a measurable function. 
If a sequence of probability measures $\mu^{N}$ on $Z_{N}$ satisfies that
$$
\lim_{N \to \infty}\mu^{N}\left[\left|\frac{1}{|V_{N}|}\sum_{x \in V_{N}}J\left(\Phi_{N}(x)\right)\eta^{N}_{x}(0)- \int_{\mathbb{T}^{d}}J(u)\rho_{0}(u)du \right| > \delta \right]=0,
$$
for every $\delta >0$ and for every continuous functions $J:\mathbb{T}^{d} \to \mathbb{R}$,
then for every $t \ge 0$, 
$$
\lim_{N \to \infty}\mathbb{P}_{N}^{H}\left[\left|\frac{1}{|V_{N}|}\sum_{x \in V_{N}}J\left(\Phi_{N}(x)\right)\eta^{N}_{x}(t)- \int_{\mathbb{T}^{d}}J(u)\rho(t,u)du \right| > \delta \right]=0,
$$
for every $\delta >0$ and for every continuous functions $J:\mathbb{T}^{d} \to \mathbb{R}$,
where $\rho(t,u)$ is the unique weak solution of the following quasi-linear parabolic equation:
\begin{equation}\label{pde1}
\frac{\partial}{\partial t}\rho = \nabla \mathbb{D} \nabla \rho - \frac{1}{2|V_{0}|}\sum_{e \in E_{0}}\nabla_{{\bf v}(e)}\left(\rho(1-\rho)\nabla_{{\bf v}(e)}H \right), \ \ \ \ \ \rho(0,\cdot)=\rho_{0}(\cdot).
\end{equation}
\end{thm}

Here we define $\nabla_{{\bf v}(e)}:=\sum_{i=1}^{d}v_{i}(e)(\partial/\partial x_{i})$ for $e \in E_{0}$.

We give examples corresponding to ones in Section \ref{harmonic}.

{\it Examples}

\begin{enumerate}
\item[0.]
The one dimensional standard lattice.

For the embedding in Example 0a., we recover the equation in Theorem 3.1 in \cite{KOV}:
$$
\frac{\partial}{\partial t}\rho = \frac{1}{2}\frac{\partial^{2}}{\partial x^{2}}\rho-\frac{\partial}{\partial x}\left(\rho(1-\rho)\frac{\partial H}{\partial x}\right).
$$

For the embedding in Example 0b., we have the following equation:
$$
\frac{\partial}{\partial t}\rho = \frac{5}{4}\frac{\partial^{2}}{\partial x^{2}}\rho-\frac{5}{2}\frac{\partial}{\partial x}\left(\rho(1-\rho)\frac{\partial H}{\partial x}\right).
$$

\item[1.]
The square lattice.

For the square lattice and its embedding in Example 1a., we have the following equation:
$$
\frac{\partial}{\partial t}\rho = \frac{1}{2}\left(\frac{\partial^{2}}{\partial x^{2}}+\frac{\partial^{2}}{\partial y^{2}}\right)\rho-\frac{\partial}{\partial x}\left(\rho(1-\rho)\frac{\partial H}{\partial x}\right)-\frac{\partial}{\partial x}\left(\rho(1-\rho)\frac{\partial H}{\partial y}\right).
$$

For the square lattice and its embedding in Example 1b., we have the following equation:

\begin{align*}
&\frac{\partial}{\partial t}\rho =\left(\frac{5}{8}\frac{\partial^{2}}{\partial x^{2}}+\frac{1}{4}\frac{\partial^{2}}{\partial x \partial y}++\frac{1}{4}\frac{\partial^{2}}{\partial y \partial x}+\frac{1}{2}\frac{\partial^{2}}{\partial y^{2}}\right)\rho  \\
&-\frac{5}{4}\frac{\partial}{\partial x}\left(\rho(1-\rho)\frac{\partial H}{\partial x}\right) -\frac{1}{2}\frac{\partial}{\partial x}\left(\rho(1-\rho)\frac{\partial H}{\partial y}\right)
-\frac{1}{2}\frac{\partial}{\partial y}\left(\rho(1-\rho)\frac{\partial H}{\partial x}\right)
-\frac{\partial}{\partial y}\left(\rho(1-\rho)\frac{\partial H}{\partial y}\right).
\end{align*}

\item[2.]
The hexagonal lattice, the Kagom\'{e} lattice.

For the hexagonal lattice, the Kagom\'{e} lattice and their embeddings in Example 2. and 3., we have the following same equation:
$$
\frac{\partial}{\partial t}\rho =\frac{3}{8}\left(\frac{\partial^{2}}{\partial x^{2}}+\frac{\partial^{2}}{\partial y^{2}}\right)\rho 
-\frac{3}{4}\frac{\partial}{\partial x}\left(\rho(1-\rho)\frac{\partial H}{\partial x}\right) -\frac{3}{4}\frac{\partial}{\partial y}\left(\rho(1-\rho)\frac{\partial H}{\partial y}\right).
$$

\end{enumerate}

\section{Replacement theorem}\label{replacement theorem}

In this section, we formulate the replacement theorem and give its proof. The replacement theorem is given by the form of super exponential estimate and follows from the one-block estimate and the two blocks estimate. 

\subsection{Local function bundles}\label{local function bundles}

For our purpose, 
we introduce local function bundles which describe the local interactions of particles and the two types of local averages for local function bundles.

\begin{Def}
A local function bundle $f$ on $V\times Z$ is a function $f:V\times Z \to \mathbb{R}$, which satisfies that for each $z \in V$ there exists $r>0$ such that $f_{z}:Z \to \mathbb{R}$ depends only on $\{\eta_{x} \ | \ d(x,z) \le r\}$.
Here $d$ is the graph distance in $X$.
We say that a local function bundle $f:V \times Z \to \mathbb{R}$ is $\Gamma$-periodic if it holds that $f_{\sigma x}(\sigma \eta)=f_{x}(\eta)$ for any $\sigma \in \Gamma$, $x \in V$ and $\eta \in Z$.
\end{Def}

Here we give examples of $\Gamma$-periodic local function bundles on $V \times Z$.
We use the first one and the third one later.

{\it Examples}
\begin{itemize}
\item If we define $f:V \times Z \to \mathbb{R}$ by for $x \in V$ and $\eta \in Z$
$$
f_{x}(\eta)=\eta_{x},
$$
then $f$ is a $\Gamma$-periodic local function bundle on $V \times Z$.

\item If we define $f:V \times Z \to \mathbb{R}$ by for $x \in V$ and $\eta \in Z$
$$
f_{x}(\eta)=\prod_{e \in E_{x}}\eta_{te},
$$
then $f$ is a $\Gamma$-periodic local function bundle on $V \times Z$.

\item Fix $e \in E^{0}$.
If we define $f^{(e)}:V \times Z \to \mathbb{R}$ by
$$
f^{(e)}_{x}(\eta)=
\begin{cases}
\eta_{o\sigma e}\eta_{t\sigma e}	& \text{there exists (unique) $\sigma \in \Gamma$ such that $x=o\sigma e$,} \\
0				& \text{otherwise,}
\end{cases}	
$$
then $f$ is a $\Gamma$-periodic local function bundle on $V \times Z$.
\end{itemize}

Note that a $\Gamma$-periodic local function bundle $f:V \times Z \to \mathbb{R}$ induces a map $\underline f : V_{N} \times Z_{N} \to \mathbb{R}$ for large enough $N$ in the natural way. 
To abuse the notation, we indicate the induced map by the same character $f$.

First, for $R >0$, we define the $R$-ball by
$$
B(D_{x_{0}},R):=\bigcup_{\sigma \in \Gamma, |\sigma|\le R}\sigma D_{x_{0}} \subset V.
$$
Here $|\cdot|$ is the word metric appearing in Section \ref{approximation}.
We regard that $B(D_{x_{0}}, R)\subset V_{N}$ via the covering map when $N$ is large enough for $R$.
For a local function bundle $f:V\times Z \to \mathbb{R}$, we define the local average of $f$ on blocks $\{\sigma D_{x_{0}}\}_{\sigma \in \Gamma}$ by for $x \in V$,
$$
\overline f_{x, R}:=\frac{1}{\left|[x]B(D_{x_{0}},R)\right|}\sum_{z \in [x]B(D_{x_{0}},R)} f_{z} : Z \to \mathbb{R},
$$
where $[x]$ is a unique element $\sigma \in \Gamma$ such that $x \in \sigma D_{x_{0}}$ and $\left|[x]B(D_{x_{0}},R)\right|$ denotes the cardinality of the set, which is equal to $\left|B(D_{x_{0}},R)\right|$. 
Note that $\overline f_{x, R}=\overline f_{x_{0}, R}$ for every $x \in D_{x_{0}}$.
As a special case, we define for $\eta \in Z$ and $x \in V$,
$$
\overline \eta_{x,R}:=\frac{1}{\left|[x]B(D_{x_{0}},R)\right|}\sum_{z \in [x]B(D_{x_{0}},R)} \eta_{z} : Z \to \mathbb{R}.
$$

Second, we define the local average of $f$ on each $\Gamma$-orbit, $\{\sigma x\}_{\sigma \in \Gamma}$ by
for $x \in V$,
$$
\widetilde f_{x, R}:=\frac{1}{\left|\{\sigma \ | \ |\sigma| \le R\}\right|} \sum_{\sigma \in \Gamma, |\sigma| \le R}f_{\sigma x}: Z \to \mathbb{R}.
$$
Note that $\overline f_{\cdot, R}$ and $\widetilde f_{\cdot, R}$ are $\Gamma$-periodic when $f$ is $\Gamma$-periodic.

If a local function bundle $f$ is $\Gamma$-periodic and $N$ is large enough,
then $\overline f_{\cdot, R}, \widetilde f_{\cdot, R}$ induce the functions on $Z_{N}$ in the natural way.
To abuse the notation, we indicate the induced maps by the same characters $\overline f_{\cdot, R}, \widetilde f_{\cdot, R}$.

\subsection{Super exponential estimate}

For a local function bundle $f:V\times Z \to \mathbb{R}$, for $x \in V$, 
let us define $\langle f_{x} \rangle (\rho):=\mathbb{E}_{\nu_{\rho}}\left[f_{x}\right]$, the expectation with respect to the Bernoulli measure $\nu_{\rho}$.

The following estimate allows us to replace the local averages of the local function bundle by the global averages of the empirical density. We call the following theorem the replacement theorem. We prove it in the form of the super exponential estimate.

\begin{thm}\label{super exponential estimate}{\upshape (Super exponential estimate of the replacement theorem)}

Fix $T>0$. For any $\Gamma$-periodic local function bundles $f:V\times Z \to \mathbb{R}$, for every $x \in D_{x_{0}}$ and for every $\delta >0$, it holds that
$$
\lim_{K \to \infty}\limsup_{\epsilon \to 0}\limsup_{N \to \infty}
\frac{1}{|\Gamma_{N}|}\log \mathbb{P}_{N}^{H}
\left(\frac{1}{|\Gamma_{N}|}\int_{0}^{T}V_{x, N, \epsilon, K}(\eta(t))dt \ge \delta \right) = - \infty,
$$
where 
$$
V_{x, N, \epsilon, K}(\eta):=\sum_{\underline \sigma \in \Gamma_{N}}\left| \widetilde f_{\underline \sigma x, K}(\eta) - \langle f_{x}\rangle (\overline \eta_{\underline \sigma x_{0}, \epsilon N})\right|.
$$
\end{thm}

Note that for every $x \in D_{x_{0}}$, $\overline \eta_{\underline \sigma x_{0}, R}=\overline \eta_{\underline \sigma x, R}$ $(R > 0)$.

We denote by $\mathbb{P}_{N}$ the corresponding distribution on $D([0,T], Z_{N})$ of continuous time Markov chain $\eta^{N}(t)$ generated by $N^{2}L_{N}$ with the initial measure $\mu^{N}$.
Furthermore, we denote by $\mathbb{P}_{N}^{eq}$ the corresponding distribution on $D([0,T], Z_{N})$ of continuous time Markov chain $\eta^{N}(t)$ generated by $N^{2}L_{N}$ with the initial measure $\nu^{N}_{1/2}$, i.e., an equilibrium measure.
We denote by $\mathbb{E}_{N}^{H}$ the expectation with respect to $\mathbb{P}_{N}^{H}$, by $\mathbb{E}_{N}$ the one with respect to $\mathbb{P}_{N}$ and by $\mathbb{E}_{N}^{eq}$ the one with respect to $\mathbb{P}_{N}^{eq}$.
For a probability measure $\mu$ on some probability space, we also denote by $\mathbb{E}_{\mu}$ the expectation with respect to $\mu$.

By the following proposition, we reduce the super exponential estimate for $\mathbb{P}_{N}^{H}$ to the one for $\mathbb{P}_{N}$.

\begin{pro}\label{Radon-Nikodym}
There exists a constant $C(H, T)>0$ such that
$$
\frac{d\mathbb{P}_{N}^{H}}{d\mathbb{P}_{N}} \le \exp C(H, T)|\Gamma_{N}|.
$$
\end{pro}

\begin{proof}
To simplify the notation, put $H(t, x)=H(t, \Phi_{N}(x))$ for $x \in V_{N}$.
To calculate the Radon-Nikodym derivative:
\begin{align*}
\frac{d\mathbb{P}_{N}^{H}}{d\mathbb{P}_{N}}	&=\exp\Bigg[\sum_{x \in V_{N}}H(T, x)\eta_{x}(T)-\sum_{x \in V_{N}}H(0, x)\eta_{x}(0) \\
& \ \ \ \ \ \ \ \ \ \ \ \ \ \ \ \ \ \ \ \ \ \ \ \ \ \ \ \ \ \ \ \ \ \ \ \ \ \ \  -\int_{0}^{T}e^{-\sum_{x \in V_{N}}H(t, x)\eta_{x}(t)}(\partial_{t}+N^{2}L_{N})e^{\sum_{x \in V_{N}}H(t, x)\eta_{x}(t)}dt\Bigg]   \\
&=\exp\Bigg[\sum_{x \in V_{N}}H(T, x)\eta_{x}(T)-\sum_{x \in V_{N}}H(0, x)\eta_{x}(0) -\int_{0}^{T}\Bigg\{\sum_{x \in V_{N}}\frac{\partial H}{\partial t}(t,x)\eta_{x}(t) \\
& \ \ \ \ \ \ \ \ \ \ \ \ \ \ + \frac{N^{2}}{4}\sum_{e \in E_{N}}\Big(\eta_{oe}(1-\eta_{te})e^{H(t,te)-H(t,oe)} + \eta_{te}(1-\eta_{oe})e^{H(t,oe)-H(t,te)}  \\
& \ \ \ \ \ \ \ \ \ \ \ \ \ \ \ \ \ \ \ \ \ \ \ \ \ \ \ \ \ \ \ \ \ \ \ \ \ \ \ \ \ \ \ \ \ \ \ \ \ \ \ \ \ \ \ \ \ \ \ \ \ \ \ \ \ \ \ \ \ \ \ \ \ \ \ \ \ \ -\eta_{oe}(1-\eta_{te}) - \eta_{te}(1-\eta_{oe})\Big)\Bigg\}dt\Bigg]   \\
&=\exp\Bigg[\sum_{x \in V_{N}}H(T, x)\eta_{x}(T)-\sum_{x \in V_{N}}H(0, x)\eta_{x}(0) 
 -\int_{0}^{T}\Big\{\sum_{x \in V_{N}}\frac{\partial H}{\partial t}(t,x)\eta_{x}(t) \\
&+ \frac{N^{2}}{2}\sum_{x \in V_{N}}\sum_{e \in E_{N,x}}\eta_{x}\left(e^{H(t,te)-H(t,oe)}-1\right)
- \frac{N^{2}}{4}\sum_{e \in E_{N}}\left(e^{H(t,te)-H(t,oe)}+e^{H(t,oe)-H(t,te)}-2\right)\eta_{oe}\eta_{te}\Big\}dt\Bigg].   \\
\end{align*}

By the inequality $|e^{z} - 1 -z -(1/2)z^{2}| \le (1/6)|z|^{3}e^{|z|}$ for $z \in \mathbb{R}$, we have that
\begin{align*}
&\Big|N^{2}\sum_{e \in E_{N, x}}\left[e^{H(t, te)-H(t, oe)}-1\right]  \\
&-N^{2}\sum_{e \in E_{N, x}}\left(H(t, te)- H(t, oe)\right) - \frac{1}{2}N^{2}\sum_{e \in E_{N, x}}\left(H(t, te)- H(t, oe)\right)^{2}\Big|  \\
&\le \frac{1}{6}N^{2}\sum_{e \in E_{N, x}}\left|H(t, te)- H(t, oe)\right|^{3}e^{\left|H(t, te)- H(t, oe)\right|},
\end{align*}
and thus, since $\Phi$ is harmonic,
\begin{align*}
&N^{2}\sum_{e \in E_{N, x}}\left[e^{H(t, te)-H(t, oe)}-1\right]	\\
&=\frac{1}{2}\sum_{e \in E_{N, x}}\sum_{i,j=1}^{d}\frac{\partial^{2}H}{\partial x_{i} \partial x_{j}}(t, x)v_{i}(e)v_{j}(e)
+\frac{1}{2}\sum_{e \in E_{N, x}}\left(\nabla_{{\bf v}(e)}H(t, x)\right)^{2} + o_{N}.
\end{align*}

Furthermore, 
$$
N^{2}(e^{H(t, te)-H(t, oe)}+e^{H(t, oe)-H(t, te)}-2)=\left(\nabla_{{\bf v}(e)}H(t,oe)\right)^{2} + o_{N}.
$$

Hence, there exists a constant $C(H, T)>0$ depending only on $H$ and $T$ such that $d\mathbb{P}^{H}_{N}/d\mathbb{P_{N}}$ is bounded from above by $\exp C(H, T)|\Gamma_{N}|$.
It completes the proof.
\end{proof}

The super exponential estimate for $\mathbb{P}_{N}$ induces the one for $\mathbb{P}_{N}^{H}$ since it holds that for any Borel sets $A \subset D([0, T], Z_{N})$, $\mathbb{P}_{N}^{H}(A)\le \left(\exp C(H, T)|\Gamma_{N}|\right)\mathbb{P}_{N}(A)$.
Furthermore, it is enough to prove the super exponential estimate for $\mathbb{P}_{N}^{eq}$ since for any Borel sets $A \subset D([0, T], Z_{N})$, $\mathbb{P}_{N}(A) \le 2^{|V_{N}|} \mathbb{P}_{N}^{eq}(A)$.

Let $\mathcal{P}(Z_{N}), \mathcal{P}(Z)$ be the spaces of probability measures on $Z_{N}, Z$, respectively. Define $\nu^{N}:=\otimes_{V_{N}}\nu^{1}_{1/2}$, $\nu:=\otimes_{V}\nu^{1}_{1/2}$ the $(1/2)$-Bernoulli measure on $Z_{N}$, $Z$, respectively. Here we introduce a functional on $\mathcal{P}(Z_{N})$, which is the Dirichlet form for the density function.

\begin{Def}
For $\mu \in \mathcal{P}(Z_{N})$, put the density $\phi=d\mu/d\nu^{N}$. 
The Dirichlet form of $\sqrt \phi$ is defined by
$$
I_{N}(\mu)=-\int_{Z_{N}}\sqrt \phi L_{N} \sqrt \phi d\nu^{N}.
$$
\end{Def}

Note that
$I_{N}(\mu)=(1/4)\int_{Z_{N}}\sum_{e \in E_{N}}(\pi_{e}\sqrt \phi)^{2}d\nu^{N} \ge 0$.

{\it Remark.}

The functional $I_{N}$ is also called the $I$-functional in the different literatures.

Let us define the subset of the space of probability measures on $Z_{N}$ by for $C > 0$,
$$
\mathcal{P}_{N,C}:=\left\{\mu \in \mathcal{P}(Z_{N}) \ | \ \text{$\mu$ is $\Gamma_{N}$-invariant and $I_{N}(\mu) \le C\frac{|V_{N}|}{N^{2}}$}\right\}.
$$

The proof of the super exponential estimate for $\mathbb{P}_{N}^{eq}$ is reduced to the following:
\begin{thm}\label{local ergodic theorem}
For every $C > 0$ and every $x \in D_{x_{0}}$,
$$
\lim_{K \to \infty}\limsup_{\epsilon \to 0}\limsup_{N \to \infty} \sup_{\mu \in \mathcal{P}_{N,C}}\mathbb{E}_{\mu}\left|\widetilde f_{x, K} - \langle f_{x}\rangle (\overline \eta_{x_{0}, \epsilon N}) \right|=0.
$$
\end{thm}

First, we prove Theorem\ref{super exponential estimate} by using Therem\ref{local ergodic theorem}.

\begin{proof}[Proof of Theorem\ref{super exponential estimate}]
Fix any $T >0$.
By the above argument, it is enough to show that
for any $\Gamma$-periodic local function bundles $f:V\times Z \to \mathbb{R}$, for every $x \in D_{x_{0}}$ and for every $\delta >0$, 
$$
\lim_{K \to \infty}\limsup_{\epsilon \to 0}\limsup_{N \to \infty}
\frac{1}{|\Gamma_{N}|}\log \mathbb{P}_{N}^{eq}
\left(\frac{1}{|\Gamma_{N}|}\int_{0}^{T}V_{x, N, \epsilon, K}(\eta(t))dt \ge \delta \right) = - \infty,
$$
where 
$$
V_{x, N, \epsilon, K}(\eta):=\sum_{\underline \sigma \in \Gamma_{N}}\left| \widetilde f_{\underline \sigma x, K}(\eta) - \langle f_{x}\rangle (\overline \eta_{\underline \sigma x_{0}, \epsilon N})\right|.
$$

By the Chebychev inequality, for every  $a >0$ and every $\delta >0$,
$$
\mathbb{P}_{N}^{eq}\left(\frac{1}{|\Gamma_{N}|}\int_{0}^{T}V_{x, N, \epsilon, K}dt \ge \delta \right)
\le \mathbb{E}_{N}^{eq}\exp\left(a \int_{0}^{T}V_{x, N, \epsilon, K}dt - a \delta |\Gamma_{N}| \right).
$$

Now, the operator $N^{2}L_{N} + a V_{x, N, \epsilon, K}$ acting on $L^{2}(Z_{N}, \nu^{N})$ is self-adjoint for all $a >0$ and all $x \in D_{x_{0}}$.
Let $\lambda_{x, N, \epsilon, K}(a)$ be the largest eigenvalue of $N^{2}L_{N} + a V_{x, N, \epsilon, K}$. By using the Feynman-Kac formula, 
$$
\mathbb{E}_{N}^{eq}\exp\left(a \int_{0}^{T}V_{x, N, \epsilon, K}dt\right) \le \exp T\lambda_{x, N, \epsilon, K}(a).
$$

Therefore, it is suffice to show that for every $a >0$,
\begin{equation}\label{eigenvalue}
\lim_{K \to \infty}\limsup_{\epsilon \to 0}\limsup_{N \to \infty}\frac{1}{|\Gamma_{N}|}\lambda_{x, N, \epsilon, K}(a)=0,
\end{equation}
since (\ref{eigenvalue}) implies that
$$
\lim_{K \to \infty}\limsup_{\epsilon \to 0}\limsup_{N \to \infty}
\frac{1}{|\Gamma_{N}|}\log \mathbb{P}_{N}^{eq}
\left(\frac{1}{|\Gamma_{N}|}\int_{0}^{T}V_{x, N, \epsilon, K}(\eta(t))dt \ge \delta \right) \le -a \delta,
$$
and we obtain the theorem by taking $a$ to the infinity.

By the variational principle, the largest eigenvalue $\lambda_{x, N, \epsilon, K}(a)$ is represented by the following:
$$
\lambda_{x, N, \epsilon, K}(a)=\sup_{\mu \in \mathcal{P}(Z_{N})}\left\{a\int_{Z_{N}}V_{x, N, \epsilon, K} d\mu - N^{2}I_{N}(\mu)\right\}.
$$
See \cite{KL} Appendix 3 for more details.
 
Denote the average of $\mu$ by $\Gamma_{N}$-action by
$
\widetilde \mu :=(1/|\Gamma_{N}|)\sum_{\underline \sigma \in \Gamma_{N}}\mu \circ \underline \sigma.
$
Then $\widetilde \mu$ is $\Gamma_{N}$-invariant so that
$$
\frac{1}{|\Gamma_{N}|}\int_{Z_{N}}V_{x, N, \epsilon, K}d\mu = \frac{1}{|\Gamma_{N}|}\mathbb{E}_{\mu}\left[V_{x, N, \epsilon, K}\right]=\mathbb{E}_{\widetilde \mu}\left|\widetilde f_{x, K} - \langle f_{x}\rangle(\overline \eta_{x_{0}, \epsilon N})\right|.
$$
The functional $I_{N}(\cdot)$ is also $\Gamma_{N}$-invariant, i.e., $I_{N}(\mu \circ \underline \sigma)=I_{N}(\mu)$ for $\mu \in \mathcal{P}(Z_{N})$, $\underline \sigma \in \Gamma_{N}$. 
Thus, it is suffice to consider $\Gamma_{N}$-invariant measures $\mu$ to estimate the largest eigenvalue $\lambda_{x, N, \epsilon, K}(a)$.
Furthermore, it is suffice to consider the case where $\mu$ satisfies
$
\int_{Z_{N}}V_{x, N, \epsilon, K}d\mu \ge N^{2}I_{N}(\mu).
$
There exists a constant $C(f)>0$ depending on $f$ such that $V_{x, N, \epsilon, K} \le C(f)|\Gamma_{N}|$, thus we reduce $\mu \in \mathcal{P}(Z_{N})$ to every $\Gamma_{N}$-invariant measure satisfying that for every $C >0$,
$
I_{N}(\mu)\le C|V_{N}|/N^{2}.
$
This shows that we can reduce to $\mathcal{P}_{N,C}$ for every $C>0$, and thus it is enough to show for every $C >0$ and every $x \in D_{x_{0}}$,
$$
\lim_{K \to \infty}\limsup_{\epsilon \to 0}\limsup_{N \to \infty}\sup_{\mu \in \mathcal{P}_{N, C}}\mathbb{E}_{\mu}\left|\widetilde f_{x, K} - \langle f_{x} \rangle (\overline \eta_{x_{0}, \epsilon N}) \right|=0,
$$
to obtain (\ref{eigenvalue}). This follows from Theorem\ref{local ergodic theorem}. It completes the proof.
\end{proof}

\subsection{The one-block estimate}\label{section of the one-block estimate}

In this section, we prove the one-block estimate.
We regard a probability measure $\mu$ on $Z_{N}$ as one on $Z$ by periodic extension. 
Let $pr_{N}:V \to V_{N}$ be the covering map by $N\Gamma$-action, and define the periodic inclusion
$i^{N}_{per}:Z_{N} \hookrightarrow Z$ by $(i^{N}_{per}\eta)_{z}:=\eta_{pr_{N}(z)}, \eta \in Z_{N}, z \in V$.
We identify $\mu$ on $Z_{N}$ with its push forward by $i^{N}_{per}$. 
On the other hand, we identify an $N\Gamma$-invariant probability measure on $Z$ with a probabiltiy measure on $Z_{N}$.

First, for a finite subgraph $\Lambda=(V_{\Lambda},E_{\Lambda})$ of X, we define the restricted state space $Z_{\Lambda}:=\{0,1\}^{V_{\Lambda}}$ and the $(1/2)$-Bernoulli measure by $\nu^{\Lambda}:=\otimes_{V_{\Lambda}}\nu^{1}_{1/2}$ on $Z_{\Lambda}$. 
Let us define the operator acting on $L^{2}(Z, \nu)$ by $L^{\circ}_{\Lambda}:=(1/2)\sum_{e \in E_{\Lambda}}\pi_{e}$.
For $\mu \in \mathcal{P}(Z)$, $\mu |_{\Lambda}$ stands for the restriction of $\mu$ on $Z_{\Lambda}$ and $\phi_{\Lambda}:=d\mu|_{\Lambda}/d\nu^{\Lambda}$ its density. 
We also define the corresponding Dirichlet form of $\sqrt \phi_{\Lambda}$ by
$$
I^{\circ}_{\Lambda}(\mu):= - \int_{Z_{\Lambda}} \sqrt{\phi_{\Lambda}}L^{\circ}_{\Lambda}\sqrt{\phi_{\Lambda}}d\nu^{\Lambda}.
$$
For large enough $N$, we regard $\Lambda$ as a subgraph of $X_{N}$ by taking a suitable fundamental domain in $V$ for $N\Gamma$-action.

By the convexity of the Dirichlet form,
$$
I^{\circ}_{\Lambda}(\mu) \le \frac{1}{4}\sum_{e \in E_{\Lambda}}\int_{Z_{N}}(\pi_{e}\sqrt{\phi})^{2}d\nu^{N},
$$
by putting $\phi:=d\mu/d\nu^{N}$. 

The one-block estimate is stated as follows:

\begin{thm}\label{the one-block estimate}{\upshape (The one-block estimate.)}
For every $\Gamma$-periodic local function bundles, $f:V\times Z \to \mathbb{R}$, every $x \in D_{x_{0}}$ and every $C >0$,
$$
\lim_{K \to \infty}\limsup_{N \to \infty}\sup_{\mu_{N} \in \mathcal{P}_{N, C}}\mathbb{E}_{\mu_{N}}\left| \widetilde f_{x, K} - \langle f_{x} \rangle \left(\overline \eta_{x_{0}, K}\right) \right|=0.
$$
\end{thm}

\begin{proof}
For any probability measure $\mu_{N} \in \mathcal{P}_{N, C}$,
we apply the above argument by setting $\Lambda$ as 
$V_{\Lambda}:=B(D_{x_{0}},K)$ and $E_{\Lambda}:=\{e \in E \ | \ o(e), t(e) \in V_{\Lambda} \}$.
Since $\mu_{N}$ and $\nu^{N}$ are $\Gamma_{N}$-invariant,
\begin{align*}
I^{\circ}_{\Lambda}(\mu_{N})		&\le \frac{1}{4}\sum_{\underline \sigma \in \Gamma_{N}, |\underline \sigma| \le K}\sum_{e \in \underline \sigma E^{0}}\int_{Z_{N}}(\pi_{e}\sqrt{\phi})^{2}d\nu^{N}	\\
											&\le |\{\underline \sigma \in \Gamma_{N} \ | \ |\underline \sigma| \le K\}|\cdot \frac{1}{|\Gamma_{N}|}I_{N}(\mu_{N}).
\end{align*}

Since $\mu_{N} \in \mathcal{P}_{N, C}$ and $|V_{N}|/|\Gamma_{N}|=|V_{0}|$, it holds that 
$$
I^{\circ}_{\Lambda}(\mu_{N})	\le  (2K)^{d}\cdot C\frac{|V_{0}|}{N^{2}} \to 0 \ \ \ \ \ \text{as $N \to \infty$}.
$$

We note that $\mathcal{P}(Z)$ is compact with respect to the weak topology, and thus $\{\mu_{N}\} \subset \mathcal{P}(Z)$ has a subsequence which convergences to some $\mu$ in $\mathcal{P}(Z)$.
Let $\mathcal{A}$ be the set of all limit points of $\{\mu_{N}\}$ in $\mathcal{P}(Z)$.
By the above argument,
$I^{\circ}_{\Lambda}(\mu)=0$ for all $\mu \in \mathcal{A}$.
Therefore, since $I^{\circ}_{\Lambda}(\mu)=(1/4)\sum_{e \in E_{\Lambda}}\sum_{\eta \in Z_{\Lambda}}(\pi_{e}\sqrt{\mu|_{\Lambda}(\eta)})^{2}=0$ for every $\mu \in \mathcal{A}$ by the definition of $I^{\circ}_{\Lambda}$, we obtain that $\mu|_{\Lambda}(\eta^{e})=\mu|_{\Lambda}(\eta)$ for every $e \in E_{\Lambda}$ and every $\eta \in Z_{\Lambda}$.
This shows that random variables $\eta_{x}, x \in V$ are exchangeable under $\mu$.
By the de Finetti theorem, there exists a probability measure $\lambda$ on $[0,1]$ such that $\mu=\int_{0}^{1}\nu_{\rho}\lambda(d\rho)$. 
Since
$$
\limsup_{N \to \infty}\sup_{\mu_{N} \in \mathcal{P}_{N, C}}\mathbb{E}_{\mu_{N}}\left| \widetilde f_{x, K} - \langle f_{x} \rangle \left(\overline \eta_{x_{0}, K}\right) \right| 
\le
\sup_{\mu \in \mathcal{A}}\mathbb{E}_{\mu}\left| \widetilde f_{x, K} - \langle f_{x} \rangle \left(\overline \eta_{x_{0}, K}\right) \right|,
$$
it is enough to show that
$$
\limsup_{K \to \infty}\sup_{\rho \in [0,1]}\mathbb{E}_{\nu_{\rho}}\left| \widetilde f_{x, K} - \langle f_{x} \rangle \left(\overline \eta_{x_{0}, K}\right) \right|=0,
$$
for every $\Gamma$-invariant local function bundles $f$.

By the definition of the $\Gamma$-periodic local function bundle, there exists a constant $L > 0$ such that for every $x \in V$, $f_{x}:Z \to \mathbb{R}$ depends on at most $\{\eta_{z} \ | \ z \in [x]B(D_{x_{0}},L)\}$.
Therefore we obtain that there exists a constant $C(f)>0$ depending only on $f$ such that
$$
\mathbb{E}_{\nu_{\rho}}\left[\widetilde f_{x,K}-\mathbb{E}_{\nu_{\rho}}\left[\widetilde f_{x,K}\right]\right]^{2}\le C(f)\cdot \frac{L^{d}}{K^{d}} \to 0 \ \ \ \ \ \text{as $K \to \infty$}.
$$

Note that $\langle f_{x} \rangle(\rho)=\mathbb{E}_{\nu_{\rho}}\left[\widetilde f_{x,K}\right]$ for every $x \in D_{x_{0}}$ since the Bernoulli measure $\nu_{\rho}$ is $\Gamma_{N}$-invariant.
In addition, we also obtain that there exists a constant $C' >0$ not depending on $\rho$, 
$$
\mathbb{E}_{\nu_{\rho}}\left|\overline \eta_{x_{0},K} - \rho \right|^{2} \le \frac{C'}{K^{d}} \to 0 \ \ \ \ \ \text{as $K \to \infty$}.
$$
Finally, since $\langle f_{x} \rangle(\rho)$ is a polynomial with respect to $\rho$, in particular, uniformly continuous on $[0,1]$, we obtain that
$$
\lim_{K \to \infty}\sup_{\rho \in [0,1]} \mathbb{E}_{\nu_{\rho}}\left| \widetilde f_{x, K} - \langle f_{x} \rangle(\overline \eta_{x_{0}, K}) \right|=0 \ \ \ \ \ \text{for every $x \in D_{x_{0}}$}.
$$
This concludes the theorem
$$
\lim_{K \to \infty}\limsup_{N \to \infty}\sup_{\mu_{N} \in \mathcal{P}_{N, C}} \mathbb{E}_{\mu_{N}}\left| \widetilde f_{x, K} - \langle f_{x} \rangle(\overline \eta_{x_{0}, K}) \right|=0 \ \ \ \ \ \text{for every $x \in D_{x_{0}}$}.
$$
\end{proof}

\subsection{The two-blocks estimate}

In this section, we prove the two-blocks estimate.
We identify a probability measure on $Z_{N}$ with its periodic extension on $Z$ in the same manner as Section\ref{section of the one-block estimate}.

\begin{thm}(The two-blocks estimate)\label{the two-blocks estimate}
For every $C >0$,
$$
\lim_{K \to \infty}\limsup_{\epsilon \to 0}\limsup_{L \to \infty}\limsup_{N \to\infty}\sup_{\sigma \in \Gamma s.t. L<|\sigma|\le \epsilon N} \sup_{\mu_{N} \in \mathcal{P}_{N,C}}\mathbb{E}_{\mu_{N}}\left|\overline \eta_{x_{0},K} - \overline \eta_{\sigma x_{0},K} \right|=0.
$$
\end{thm}

Let us denote by $\mathcal{P}(Z\times Z)$ the space of probability measures on $Z \times Z$.
We define the map for $\sigma \in \Gamma$,
$\hat \sigma : Z \to Z\times Z$
by $\hat \sigma(\eta):=(\eta,\sigma \eta)$, $\eta \in Z$. 
For $\mu \in \mathcal{P}(Z)$, we define the push forward of $\mu$ by $\hat \sigma$ via $\hat \sigma \mu:=\mu \circ \hat \sigma^{-1} \in \mathcal{P}(Z\times Z)$. 

Let us denote by $\mathcal{A}^{\circ}_{\epsilon,L} \subset \mathcal{P}(Z \times Z)$ the set of all limit points of $\{\hat \sigma \mu_{N} \ | \ L < |\sigma| \le \epsilon N, \mu_{N} \in \mathcal{P}_{N,C}\}$ in $\mathcal{P}(Z \times Z)$ as $N \to \infty$ and by $\mathcal{A}^{\circ}_{\epsilon} \subset \mathcal{P}(Z \times Z)$ the set of all limit points of $\mathcal{A}^{\circ}_{\epsilon}$ as $L \to \infty$.
We put $(x_{0}, x'_{0}) \in V\times V$, where $x'_{0}$ stands for a copy of $x_{0}$. 
Then, it holds that
$$
\int_{Z}\left|\overline \eta_{x_{0},K} - \overline \eta_{\sigma x_{0},K} \right| \mu(d\eta)=\int_{Z\times Z}\left|\overline \eta_{x_{0},K} - \overline \eta'_{x'_{0},K}\right|(\hat \sigma \mu)(d\eta d\eta').
$$
Note that
\begin{align*}
\limsup_{L \to \infty}\limsup_{N \to\infty}\sup_{\sigma \in \Gamma s.t. L<|\sigma|\le \epsilon N} &\sup_{\mu \in \mathcal{P}_{N,C}}\int_{Z \times Z}\left|\overline \eta_{x_{0},K} - \overline \eta'_{x'_{0},K} \right|(\hat \sigma \mu)(d\eta d\eta')	\\
&\le
\sup_{\mu \in \mathcal{A}^{\circ}_{\epsilon}}\int_{Z \times Z}\left|\overline \eta_{x_{0},K} - \overline \eta'_{x'_{0},K} \right| \mu(d\eta d\eta').
\end{align*}

We introduce two types of generators acting on $L^{2}(Z\times Z, \nu\otimes \nu)$ and the corresponding Dirichlet forms. The first one is used for treating two different states at the same time independently. The second one is used for treating exchanges of particles between two different states.

As in Section\ref{section of the one-block estimate}, we define a subgraph
$\Lambda=(V_{\Lambda}, E_{\Lambda})$ of $X$ 
by setting 
$V_{\Lambda}:=B(D_{x_{0}},K), E_{\Lambda}:=\{e \in E \ | \ o(e), t(e) \in V_{\Lambda} \}$
and the operator acting on $L^{2}(Z, \nu)$ by
$$
L^{\circ}_{\Lambda}:=\frac{1}{2}\sum_{e \in E_{\Lambda}} \pi_{e}.
$$

For $\mu \in \mathcal{P}(Z \times Z)$, we denote by $\mu|_{\Lambda \times \Lambda}$ the restriction of $\mu$ on $Z_{\Lambda} \times Z_{\Lambda}$.
Define the Dirichlet form of $\sqrt{\phi_{\Lambda \times \Lambda}}$ by
$$
I^{\circ}_{\Lambda \times \Lambda}(\mu):=- \int_{Z_{\Lambda}\times Z_{\Lambda}}\sqrt{\phi_{\Lambda \times \Lambda}}(L^{\circ}_{\Lambda}\otimes 1 + 1\otimes L^{\circ}_{\Lambda})\sqrt{\phi_{\Lambda \times \Lambda}}d(\nu^{\Lambda}\otimes \nu^{\Lambda}).
$$

Let us introduce the notation which describes exchanges of states for $(\eta, \eta') \in Z \times Z$. 
For $(x, y) \in V \times V$, 
$(\eta, \eta')^{(x,y)} \in Z \times Z$ is the configuration obtained by exchanging values $\eta_{x}$ and $\eta'_{y}$, i.e.,
$(\eta, \eta')^{(x,y)}:=(\zeta_{z}, \zeta'_{z'})_{(z,z') \in V \times V}$ is defined by setting

$$
\zeta_{z}:=
\begin{cases}
\eta_{y}' & z=x \\
\eta_{z} & z\neq x,
\end{cases}
$$

$$
\zeta_{z}':=
\begin{cases}
\eta_{x} & z=y \\
\eta_{z}' & z\neq y.
\end{cases}
$$

Moreover, for $F \in L^{2}(Z \times Z, \nu\otimes \nu)$, we define
$\pi_{x,y}F\left((\eta, \eta')\right):=F\left((\eta, \eta')^{(x,y)}\right)-F\left((\eta, \eta')\right)$.
We define the operator acting on $L^{2}(Z \times Z, \nu\otimes \nu)$ by
$$
L^{\circ}_{x_{0},x'_{0}}:=\pi_{x_{0},x'_{0}}.
$$

The corresponding Dirichlet form of $\sqrt{\phi_{\Lambda \times \Lambda}}$ is defined by
$$
I^{(x_{0},x'_{0})}_{\Lambda \times \Lambda}(\mu):=- \int_{Z_{\Lambda}\times Z_{\Lambda}}\sqrt{\phi_{\Lambda \times \Lambda}}L^{\circ}_{x_{0},x'_{0}}\sqrt{\phi_{\Lambda \times \Lambda}}d(\nu^{\Lambda}\otimes \nu^{\Lambda}).
$$

We prove two lemmas needed later.
The first one is easy to show, so we omit the proof.

\begin{lem}\label{quasi-iso}
There exist constants $c_{1}, c_{1}', c_{2} >0$ such that for all $\sigma \in \Gamma$,
$$
c'_{1}|\sigma|-c_{2} \le d(x_{0},\sigma x_{0}) \le c_{1}|\sigma|,
$$
where $d$ is the graph distance of $X$.
\end{lem}

For $x, y \in V_{N}$ and for $\eta \in Z_{N}$, $\eta^{(x,y)}$ is the configuration obtained by exchanging two values $\eta_{x}$ and $\eta_{y}$, i.e., 
$$
\eta^{(x,y)}_{z}:=
\begin{cases}
\eta_{y} & z=x \\
\eta_{x} & z=y \\
\eta_{z} & \text{otherwise},
\end{cases}
$$
and moreover for $x, y \in V_{N}$, we define the operator $\pi_{x,y}F(\eta):=F\left(\eta^{(x,y)}\right)-F(\eta)$ for $F \in L^{2}(Z_{N}, \nu^{N})$.
These notations also indicates ones for $Z$.

The second one is the following:

\begin{lem}\label{path}

For every $\Gamma_{N}$-periodic functions $F \in L^{2}(Z_{N}, \nu^{N})$ and every $\underline{\sigma} \in \Gamma_{N}$,
$$
\int_{Z_{N}}(\pi_{x_{0},\underline{\sigma} x_{0}}F)^{2}d\nu^{N} \le 4 d(x_{0},\underline{\sigma} x_{0})^{2}\sum_{e \in E^{0}}\int_{Z_{N}}(\pi_{e}F)^{2}d\nu^{N}.
$$

\end{lem}

\begin{proof}

For $x_{0},\underline{\sigma} x_{0} \in V_{N}$, there exists a path $c=(e_{1},\dots, e_{l})$ such that $x_{0}=o(e_{1}), t(e_{l})=\underline{\sigma} x_{0}$ and $d(x_{0}, \underline{\sigma} x_{0})=l$.
Define a sequence of edges $\widetilde c=(f_{1}, \dots, f_{l}, f_{l+1}, \dots, f_{2l+1})$ by setting
$\widetilde c=(e_{1}, \dots, e_{l}, \overline e_{l-1}, \dots, \overline e_{1})$.
For $\eta \in Z_{N}$, let us define $\eta_{(0)}:=\eta, \eta_{(i)}:=\eta_{(i-1)}^{f_{i}}, 1 \le i \le 2l-1$, inductively. 
We note that $\eta^{x_{0}, \sigma x_{0}}=\eta_{(2l-1)}$. 
Then we have that
$$
\left(\pi_{x_{0},\sigma x_{0}}F(\eta)\right)^{2}	=\left(F\left(\eta^{x_{0}, \sigma x_{0}}\right)-F\left(\eta \right)\right)^{2}  
									=\left(\sum_{i=1}^{2l-1}\pi_{f_{i}}F\left(\eta_{(i-1)}\right)\right)^{2} 
									\le (2l-1)\sum_{i=1}^{2l-1}\left(\pi_{f_{i}}\left(\eta_{(i-1)}\right)\right)^{2}.
$$
If $F$ is $\Gamma_{N}$-periodic, then for each $i=1, \dots, 2l-1$, 
$$
\int_{Z_{N}}\left(\pi_{f_{i}}F\left(\eta_{(i-1)}\right)\right)^{2}d\nu^{N} 
\le 
\sum_{e \in E^{0}}\int_{Z_{N}}\left(\pi_{e}F\right)^{2}d\nu^{N}.
$$
Therefore
$$
\int_{Z_{N}}\left(\pi_{x_{0}, \sigma x_{0}}F \right)^{2}d\nu^{N}	\le (2l-1)^{2} \sum_{e \in E^{0}}\int_{Z_{N}}\left(\pi_{e}F\right)^{2}d\nu^{N}		
												\le 4d(x_{0}, \underline{\sigma} x_{0})^{2}\sum_{e \in E^{0}}\int_{Z_{N}}\left(\pi_{e}F\right)^{2}d\nu^{N}.
$$
It completes the proof.
\end{proof}

Let us define the subset of $\mathcal{P}(Z \times Z)$ by for a constant $\widetilde C >0$,
$$
\mathcal{A}_{\epsilon,\widetilde C}:=\left\{\mu \in P(Z \times Z) \ | \ I^{\circ}_{\Lambda \times \Lambda}(\mu)=0, I^{(x_{0},x'_{0})}_{\Lambda \times \Lambda}(\mu) \le \widetilde C \epsilon^{2}\right\}.
$$

Then we have the following lemma.

\begin{lem}\label{inclusion}
There exists a constant $\widetilde C >0$ such that
$$
\mathcal{A}_{\epsilon}^{\circ}\subset \mathcal{A}_{\epsilon, \widetilde C}.
$$
\end{lem}

\begin{proof}
As in Section\ref{section of the one-block estimate}, we define a subgraph $\Lambda:=(V_{\Lambda}, E_{\Lambda})$ of $X$ by setting $V_{\Lambda}:=B(D_{x_{0}}, K), E_{\Lambda}:=\{e \in E \ | \ o(e), t(e) \in V_{\Lambda}\}$ for $K$.
We take large enough $L, N$ for the diameter of $\Lambda$, $K$, so that
$V_{\Lambda} \cap \sigma V_{\Lambda} = \emptyset$ for $L < |\sigma|$ and $V_{\Lambda} \cup \sigma V_{\Lambda} \subset X_{N}$ for $L < |\sigma| \le \epsilon N$ by taking a suitable fundamental domain in $V$ by $N\Gamma$-action.
Take $\mu \in \mathcal{P}_{N, C}, \sigma \in \Gamma$. 
For $(\eta, \eta') \in Z_{\Lambda} \times Z_{\Lambda}$, if we define $\widetilde \eta \in Z_{\Lambda \cup \sigma \Lambda}$ by $\widetilde \eta|_{\Lambda}=\eta$ and $\sigma^{-1}\widetilde \eta|_{\Lambda}=\eta'$
then
$(\hat \sigma \mu)|_{\Lambda \times \Lambda}(\eta, \eta')=\mu|_{\Lambda \cup \sigma \Lambda}(\widetilde \eta)$.
Let us consider the operator acting on $L^{2}(Z, \nu)$ by
$$
L^{\circ}_{\Lambda \cup \sigma \Lambda}:=\frac{1}{2}\sum_{e \in E_{\Lambda \cup \sigma \Lambda}}\pi_{e}.
$$
For $\mu \in \mathcal{P}_{N, C}$, we denote the density of $\mu|_{\Lambda \cup \sigma \Lambda}$ 
by
$\phi_{\Lambda \cup \sigma \Lambda}:=d\mu|_{\Lambda \cup \sigma \Lambda}/d\nu^{\Lambda \cup \sigma \Lambda}$.
Then
$$
I^{\circ}_{\Lambda \times \Lambda}(\hat \sigma \mu)=I^{\circ}_{\Lambda \cup \sigma \Lambda}(\mu)
=-\int_{Z_{\Lambda \cup \sigma \Lambda}}\sqrt{\phi_{\Lambda \cup \sigma \Lambda}}L^{\circ}_{\Lambda \cup \sigma \Lambda}\sqrt{\phi_{\Lambda \cup \sigma \Lambda}}d\nu^{\Lambda \cup \sigma \Lambda}.
$$
For any $\mu \in \mathcal{P}_{N, C}$, we put $\phi:=d\mu/d\nu^{N}$.
By the convexity of the Dirichlet form and the $\Gamma_{N}$-invariance of $\phi$,
\begin{align*}
I^{\circ}_{\Lambda \times \Lambda}(\hat \sigma \mu)
\le \frac{1}{4}\sum_{e \in E_{\Lambda \cup \sigma \Lambda}}\int_{Z_{N}}\left(\pi_{e}\sqrt{\phi}\right)^{2}d\nu^{N}	
&\le \frac{|E_{\Lambda \cup \sigma \Lambda}|}{4}\sum_{e \in E^{0}}\int_{Z_{N}}\left(\pi_{e}\sqrt{\phi}\right)^{2}d\nu^{N}	\\
&\le  \frac{|E_{\Lambda}|}{2}\frac{1}{|\Gamma_{N}|}I_{N}(\mu).
\end{align*}
Since $I_{N}(\mu) \le C|V_{N}|/N^{2}$,
$$
I^{\circ}_{\Lambda \times \Lambda}(\hat \sigma \mu)	
\le  \frac{|E_{\Lambda}|}{2}\frac{|V_{N}|}{|\Gamma_{N}|}\frac{C}{N^{2}}= \frac{|E_{\Lambda}|}{2}|V_{0}|\frac{C}{N^{2}} \to 0 \ \ \ \ \ \text{as $N \to \infty$}.
$$
Therefore $I^{\circ}_{\Lambda \times \Lambda}(\mu_{L})=0$
for any $\mu_{L} \in \mathcal{A}^{\circ}_{\epsilon, L}$.
Furthermore $I^{\circ}_{\Lambda \times \Lambda}(\widetilde \mu)=0$ for any $\widetilde \mu \in \mathcal{A}^{\circ}_{\epsilon}$ by the continuity of the functional $I^{\circ}_{\Lambda \times \Lambda}$.

For $\mu \in \mathcal{P}_{N, C}$, by the convexity of the Dirichlet form,
$$
I^{(x_{0},x'_{0})}_{\Lambda \times \Lambda}(\hat \sigma \mu)
:=- \int_{Z_{\Lambda \cup \sigma \Lambda}}\sqrt{\phi_{\Lambda \cup \sigma \Lambda}}\pi_{x_{0}, \sigma x_{0}}\sqrt{\phi_{\Lambda \cup \sigma \Lambda}}d\nu^{\Lambda \cup \sigma \Lambda}
\le \frac{1}{2}\int_{Z_{N}}\left(\pi_{x_{0}, \sigma x_{0}}\sqrt{\phi_{\Lambda \cup \sigma \Lambda}}\right)^{2}d\nu^{N}.
$$

By the $\Gamma_{N}$-invariance of $\mu \in \mathcal{P}_{N, C}$ and by Lemma\ref{path},
$$
\int_{Z_{N}}\left(\pi_{x_{0},\underline{\sigma} x_{0}}\sqrt{\phi}\right)^{2}d\nu^{N}
\le 4 d(x_{0},\underline{\sigma} x_{0})^{2}\sum_{e \in E^{0}}\int_{Z_{N}}\left(\pi_{e}\sqrt{\phi}\right)^{2}d\nu^{N}  
\le 4 d(x_{0},\underline{\sigma} x_{0})^{2}\frac{4}{|\Gamma_{N}|}I_{N}(\mu).
$$

By Lemma\ref{quasi-iso} and the definition of $\mu \in \mathcal{P}_{N, C}$,
for all $\sigma \in \Gamma$ such that $L < |\sigma| \le \epsilon N$,
$$
I^{(x_{0}, x'_{0})}_{\Lambda \times \Lambda}(\hat \sigma \mu) \le 4 c_{1}^{2}|\sigma|^{2}\cdot \frac{4}{|\Gamma_{N}|}\cdot C\frac{|V_{N}|}{N^{2}} \le 16\epsilon^{2}c_{1}^{2}C|V_{0}|.
$$
By setting $\widetilde C=16c_{1}^{2}C|V_{0}|$ and the continuity of the $I^{\circ}_{\Lambda \times \Lambda}$, for every $\widetilde \mu \in \mathcal{A}^{\circ}_{\epsilon}$, we have that
$I^{(x_{0}, x'_{0})}_{\Lambda \times \Lambda}(\widetilde \mu) \le \epsilon^{2}\widetilde C$.
It concludes that $\mathcal{A}^{\circ}_{\epsilon} \subset \mathcal{A}_{\epsilon, \widetilde C}$.
\end{proof}

Let us prove Theorem\ref{the two-blocks estimate}.

\begin{proof}[Proof of Theorem\ref{the two-blocks estimate}]
Denote by $\mathcal{A}_{0}$
the set of all limit points of $\mathcal{A}_{\epsilon, \widetilde C}$
as $\epsilon \to 0$.
For every $\widetilde \mu_{0} \in \mathcal{A}_{0}$, it holds that $I^{\circ}_{\Lambda \times \Lambda}(\widetilde \mu_{0})=0$ 
and $I^{(x_{0}, x'_{0})}_{\Lambda \times \Lambda}(\widetilde \mu_{0})=0$
by the continuity of the functionals $I^{\circ}_{\Lambda \times \Lambda}$ and $I^{(x_{0}, x'_{0})}_{\Lambda \times \Lambda}$.
These show that for any $x, y \in V_{\Lambda} \times V_{\Lambda}$, $\pi_{x, y}\left(\widetilde \mu_{0}|_{\Lambda \times \Lambda}\right)=0$
and thus
for any $x, y \in V_{\Lambda}$ and for any $\eta, \eta' \in Z$,
$\widetilde \mu_{0}|_{\Lambda \times \Lambda}\left((\eta^{(x, y)}, \eta')\right)=\widetilde \mu_{0}|_{\Lambda \times \Lambda}\left((\eta, \eta')\right)$,
$\widetilde \mu_{0}|_{\Lambda \times \Lambda}\left((\eta, \eta'^{(x, y)})\right)=\widetilde \mu_{0}|_{\Lambda \times \Lambda}\left((\eta, \eta')\right)$
and
$\widetilde \mu_{0}|_{\Lambda \times \Lambda}\left((\eta, \eta')^{(x, y)}\right)=\widetilde \mu_{0}|_{\Lambda \times \Lambda}\left((\eta, \eta')\right)$,
i.e.,
$\widetilde \mu_{0}$ is exchangeable on $Z \times Z$.
By the de Finetti theorem there exists a probability measure $\lambda$ on $[0, 1]$ such that $\widetilde \mu_{0}=\int_{[0, 1]}\nu_{\rho}\otimes \nu_{\rho}\lambda(d\rho)$.

As in the proof of Theorem\ref{the one-block estimate}, $\lim_{K \to \infty}\sup_{\rho \in [0, 1]}\mathbb{E}_{\nu_{\rho}}\left|\overline \eta_{x_{0}, K} - \rho \right|^{2}=0$,
therefore, by the triangular inequality,
\begin{align*}
\sup_{\widetilde \mu_{0} \in \mathcal{A}_{0}}\mathbb{E}_{\widetilde \mu_{0}}\left|\overline \eta_{x_{0}, K}-\overline \eta_{x'_{0}, K}\right|
&\le
\sup_{\rho \in [0, 1]}\mathbb{E}_{\nu_{\rho}\otimes \nu_{\rho}}\left|\overline \eta_{x_{0}, K}-\overline \eta_{x'_{0}, K}\right|  \\
&\le
2 \sup_{\rho \in [0, 1]}\mathbb{E}_{\nu_{\rho}}\left|\overline \eta_{x_{0}, K}- \rho \right| \to 0 \ \ \ \ \ \text{as $K \to \infty$}.
\end{align*}
Finally, by Lemma\ref{inclusion}, $\mathcal{A}^{\circ}_{\epsilon} \subset \mathcal{A}_{\epsilon, \widetilde C}$ for some $\widetilde C >0$ and thus,
\begin{align*}
&\limsup_{\epsilon \to 0}\limsup_{L \to \infty}\limsup_{N \to \infty}\sup_{\sigma \in \Gamma s.t. L <|\sigma| \le \epsilon N}
\sup_{\mu \in \mathcal{P}_{N, C}}\mathbb{E}_{\mu}\left| \overline \eta_{x_{0}, K}-\overline \eta_{\sigma x'_{0}, K}\right|	\\
&\le
\limsup_{\epsilon \to 0}\sup_{\widetilde \mu \in \mathcal{A}^{\circ}_{\epsilon}}\mathbb{E}_{\widetilde \mu}\left|\overline \eta_{x_{0}, K}-\overline \eta'_{x'_{0}, K}\right|	
\le
\sup_{\widetilde \mu_{0} \in \mathcal{A}_{0}}\mathbb{E}_{\widetilde \mu_{0}}\left|\overline \eta_{x_{0}, K}-\overline \eta'_{x'_{0}, K}\right|	\to 0 \ \ \ \ \ \ \text{as $K \to \infty$}.
\end{align*}
This completes the theorem
$$
\lim_{K \to \infty}\limsup_{\epsilon \to 0}\limsup_{L \to \infty}\limsup_{N \to \infty}\sup_{\sigma \in \Gamma s.t. L <|\sigma| \le \epsilon N}
\sup_{\mu \in \mathcal{P}_{N, C}}\mathbb{E}_{\mu}\left| \overline \eta_{x_{0}, K}-\overline \eta_{\sigma x'_{0}, K}\right|=0.
$$
\end{proof}

\subsection{The proof of Theorem\ref{local ergodic theorem}}

Let us prove Theorem\ref{local ergodic theorem} by using the one-block estimate Theorem\ref{the one-block estimate} and the two-blocks estimate Theorem\ref{the two-blocks estimate}.

\begin{proof}[Proof of Theorem\ref{local ergodic theorem}]
First, we note that there exist positive constants $C_{0}, C_{1}$ and $C_{2}$
such that for any $\eta \in Z$,
\begin{align*}
&\left|\overline \eta_{x_{0}, \epsilon N} - \frac{1}{\left|\cup_{L <|\sigma| \le \epsilon N}\sigma D_{x_{0}}\right|}\sum_{z \in \cup_{L <|\sigma| \le \epsilon N}\sigma D_{x_{0}}}\overline \eta_{z, K}\right|   \\
&\le
C_{0}\frac{L^{d}}{(\epsilon N)^{d}} + C_{1}\frac{K^{d}\left((\epsilon N)^{d} - (\epsilon N -1)^{d}\right)}{(\epsilon N)^{d}} + C_{2}\frac{K^{d}\left((L +1)^{d}-L^{d}\right)}{(\epsilon N)^{d}},
\end{align*}
and thus there exists a constant $C(\epsilon, L, K)>0$ depending on $\epsilon, L$ and $K$ such that for any $\eta \in Z$,
$$
\left|\overline \eta_{x_{0}, \epsilon N} - \frac{1}{\left|\cup_{L <|\sigma| \le \epsilon N}\sigma D_{x_{0}}\right|}\sum_{z \in \cup_{L <|\sigma| \le \epsilon N}\sigma D_{x_{0}}}\overline \eta_{z, K}\right|   \\
\le
\frac{C(\epsilon, L, K)}{N},
$$
uniformly.
Then, since $\mu$ is $\Gamma$-invariant as a probability measure on $Z$,
\begin{align*}
\mathbb{E}_{\mu}\left|\overline \eta_{x_{0}, K} - \overline \eta_{x_{0}, \epsilon N}\right|
&\le
\mathbb{E}_{\mu}\left|\overline \eta_{x_{0}, K}-\frac{1}{\left|\cup_{L <|\sigma| \le \epsilon N}\sigma D_{x_{0}}\right|}\sum_{z \in \cup_{L <|\sigma| \le \epsilon N}\sigma D_{x_{0}}}\overline \eta_{z, K}\right| + \frac{C(\epsilon, L, K)}{N}   \\
&\le
\frac{1}{\left|\cup_{L <|\sigma| \le \epsilon N}\sigma D_{x_{0}}\right|}\sum_{z \in \cup_{L <|\sigma| \le \epsilon N}\sigma D_{x_{0}}}\mathbb{E}_{\mu}\left|\overline \eta_{x_{0}, K} - \overline \eta_{z, K}\right| + \frac{C(\epsilon, L, K)}{N}   \\
&\le \sup_{\sigma \in \Gamma, L< |\sigma| \le \epsilon N}\sup_{\mu \in \mathcal{P}_{N, C}}\mathbb{E}_{\mu}\left|\overline \eta_{x_{0}, K} - \overline \eta_{\sigma x_{0}, K}\right| + \frac{C(\epsilon, L, K)}{N},
\end{align*}
where the last inequality comes from the fact that for any $z \in \sigma D_{x_{0}}$ it holds that $\overline \eta_{\sigma x_{0}, K}=\overline \eta_{z, K}$.

Applying Theorem\ref{the two-blocks estimate}, we have that
$$
\lim_{K \to \infty}\limsup_{\epsilon \to 0}\limsup_{N \to \infty}\sup_{\mu \in \mathcal{P}_{N, C}}\mathbb{E}_{\mu}\left|\overline \eta_{x_{0}, K} - \overline \eta_{x_{0}, \epsilon N}\right|=0.
$$
For every $\Gamma$-periodic local function bundles $f$, $\langle f_{x} \rangle(\cdot)$ is uniformly continuous on $[0, 1]$.
Therefore,
$$
\lim_{K \to \infty}\limsup_{\epsilon \to 0}\limsup_{N \to \infty}\sup_{\mu \in \mathcal{P}_{N, C}}\mathbb{E}_{\mu}\left|\langle f_{x} \rangle\left(\overline \eta_{x_{0}, K}\right)- \langle f_{x} \rangle\left(\overline \eta_{x_{0}, \epsilon N}\right) \right|=0.
$$
Furthermore, applying the one-block estimate Theorem\ref{the one-block estimate}, for every $x \in D_{x_{0}}$,
$$
\lim_{K \to \infty}\limsup_{\epsilon \to 0}\limsup_{N \to \infty}\sup_{\mu \in \mathcal{P}_{N, C}}\mathbb{E}_{\mu}\left| \widetilde{f}_{x, K}- \langle f_{x} \rangle\left(\overline \eta_{x_{0}, \epsilon N}\right) \right|=0.
$$
It completes the proof of Theorem\ref{local ergodic theorem}.
\end{proof}

\section{The proof of Theorem\ref{main}}\label{the proof of the main theorem}

In this section, we prove Theorem\ref{main}.

Let $\Phi_{N}:X_{N} \to \mathbb{T}^{d}$ be the $N$-scaling map.
We define the empirical density by
$$
\xi_{N}(t, d\mu):=\frac{1}{|V_{N}|}\sum_{x \in V_{N}}\eta^{N}_{x}(t)\delta_{\Phi_{N}(x)}(d\mu),
$$
where $\delta_{z}$ is the delta measure at $z \in \mathbb{T}^{d}$.
The empirical density is the measure valued process.
We denote by $C^{1,2}([0,T] \times \mathbb{T}^{d})$ the space of continuous functions with continuous derivatives in $[0, T]$ and twice continuous derivatives in $\mathbb{T}^{d}$.
For every $J_{\cdot}(\cdot) \in C^{1,2}([0,T] \times \mathbb{T}^{d})$, we define 
$$
\langle J_{t}, \xi_{N}(t)\rangle:=\frac{1}{|V_{N}|}\sum_{x \in V_{N}}\eta^{N}_{x}(t)J_{t}\left(\Phi_{N}(x)\right).
$$
To abuse the notation, we denote the inner product in $L^{2}(\mathbb{T}^{d}, \nu^{N})$ by
$$
\langle F, G \rangle=\int_{\mathbb{T}^{d}}F G d\nu^{N} \ \ \ \ \text{for $F, G \in L^{2}(\mathbb{T}^{d}, \nu^{N})$}.
$$

Let us define the process as follows:
$$
M_{N}(t):=\langle J_{t}, \xi_{N}(t)\rangle - \langle J_{0}, \xi_{N}(0)\rangle - \int_{0}^{t}b_{N}(s)ds,
$$
where $b_{N}(t):=\langle (\partial/\partial t)J_{t}, \xi_{N}(t)\rangle + N^{2}L_{N}^{H} \langle J_{t}, \xi_{N}(t)\rangle$ and
$$
N_{N}(t):=\left|M_{N}(t)\right|^{2} - \int_{0}^{t}A_{N}(s)ds,
$$
where
$$
A_{N}(s):=\frac{N^{2}}{2}\sum_{e \in E_{N}}c^{H}(e, \eta, s)\left(\pi_{e}\langle J_{s}, \xi_{N}(s)\rangle\right)^{2}.
$$
Here $M_{N}(t), N_{N}(t)$ are martingales with respect to the filtration $\{\mathcal{F}_{t}\}_{t\ge 0}$, where $\mathcal{F}_{t}:=\sigma\left\{\eta^{N}(s) \ | \ 0 \le s \le t \right\}$ 
and it holds that
$$
\mathbb{E}_{N}^{H}\left|M_{N}(t)\right|^{2}=\mathbb{E}_{N}^{H}\int_{0}^{t}A_{N}(s)ds.
$$
Then we have the following lemma by applying the Doob inequality.

\begin{lem}\label{Doob}
$$
\lim_{N \to \infty}\mathbb{E}_{N}^{H}\left[\sup_{0 \le t \le T}\left|M_{N}(t)\right|^{2}\right]=0.
$$
\end{lem}

\begin{proof}
For $J \in C^{1, 2}([0, T] \times \mathbb{T}^{d})$ and for each $e \in E_{N}$, we have that
\begin{align*}
\pi_{e}\langle J_{s}, \xi_{N}(s)\rangle
&=\frac{1}{|V_{N}|}\sum_{x \in V_{N}}\pi_{e}\left(\eta^{N}_{x}(s)J_{s}\left(\Phi_{N}(x)\right)\right)   \\
&=\frac{1}{|V_{N}|}\left[\eta^{N}_{oe}(s) - \eta^{N}_{te}(s)\right]\left[J_{s}\left(\Phi_{N}(te)\right)-J_{s}\left(\Phi_{N}(oe)\right)\right].
\end{align*}

By the regularity of $J$ and the compactness of $\mathbb{T}^{d}$,
we note that there exists a constant $C(J)$ depending only on $J$, such that uniformly,
$$
\left|J_{s}\left(\Phi_{N}(te)\right)- J_{s}\left(\Phi_{N}(oe)\right)\right| \le \frac{C(J)}{N}.
$$
Thus
$$
A_{N}(s)	\le \frac{N^{2}}{2}\sum_{e \in E_{N}}c^{H}(e, \eta, s)\left[\frac{1}{|V_{N}|}\frac{C(J)}{N}\right]^{2}   
\le N^{2}|E_{N}|C'\frac{C(J)^{2}}{|V_{N}|^{2}N^{2}}
= \frac{C''}{|V_{N}|},
$$
where $C'$ is a constant such that $c^{H}(\cdot, \cdot, \cdot) \le C'$ and $C''=C'C(J)^{2}|E_{0}|/|V_{0}|$.
Then we obtain that
$$
\mathbb{E}_{N}^{H}\left|M_{N}(T)\right|^{2}=\mathbb{E}_{N}^{H}\int_{0}^{T}A_{N}(s)ds \le \frac{C''T}{|V_{N}|} \to 0 \ \ \ \ \ \text{as $N \to \infty$}.
$$
Applying the Doob inequality for the right continuous martingale $\{M_{N}(t)\}_{t}$, 
$$
\mathbb{E}_{N}^{H}\left[\sup_{0 \le t \le T}\left|M_{N}(T)\right|^{2}\right] \le 4 \mathbb{E}_{N}^{H}\left|M_{N}(T)\right|^{2},
$$
we conclude that $\lim_{N \to \infty}\mathbb{E}_{N}^{H}\left[\sup_{0 \le t \le T}\left|M_{N}(t)\right|^{2}\right]=0$.
\end{proof}

\subsection{Relative compactness of a sequence of probability measures}\label{relative compactness}

We denote by $\mathcal{M}:=\mathcal{M}(\mathbb{T}^{d})$ the space of nonnegative Borel measures with the total measure less than or equal to one on $\mathbb{T}^{d}$, endowed with the weak topology.
Since $\mathbb{T}^{d}$ is a compact metric space, the space of continuous functions $C(\mathbb{T}^{d})$ with the supremum norm is separable.
Fix a dense countable subset $\{J_{k}\}_{k=0}^{\infty}$ of $C(\mathbb{T}^{d})$, then the weak topology of $\mathcal{M}$ is given by the distance $d_{\mathcal{M}}(\cdot, \cdot)$ by
$$
d_{\mathcal{M}}(\mu, \mu'):=\sum_{k=0}^{\infty}\frac{1}{2^{k}}\cdot\frac{|\langle J_{k}, \mu \rangle-\langle J_{k}, \mu' \rangle|}{1+|\langle J_{k}, \mu \rangle-\langle J_{k}, \mu' \rangle|},
$$
for $\mu, \mu' \in \mathcal{M}$ where $\langle J_{k}, \mu \rangle:=\int_{\mathbb{T}^{d}}J_{k}d\mu$.
We note that $\mathcal{M}$ with the weak topology is compact.

Define the space of paths in $\mathcal{M}$ by
$$
D([0, T], \mathcal{M}):=\left\{\xi_{\cdot}: [0, T] \to \mathcal{M} \ | \ \text{$\xi$ is right continuous with left limits.}\right\},
$$
equipped with the Skorohod topology.
For a given process $\xi_{N}: [0, T] \to \mathcal{M}$
such that $\mathbb{P}_{N}^{H}\left(\xi_{N} \in D([0, T], \mathcal{M})\right)=1$,
we denote by $Q_{N}^{H}$ the distribution of $\xi_{N}$ on $D([0, T], \mathcal{M})$.

Then we show that the sequence $\{Q_{N}^{H}\}_{N}$ has a subsequential limit. 
The following proposition gives a sufficient condition for this.
See \cite{KL} Section 4, Theorem 1.3 for the proof.

\begin{pro}\label{relative compactness}
If for every $J_{k}$, $k=0, 1, \dots$, and every $\delta >0$,
$$
\lim_{\gamma \to 0}\limsup_{N \to \infty} Q_{N}^{H}\left(\sup_{|t-s| \le \gamma}\Big|\langle J_{k}, \xi_{N}(t)\rangle -\langle J_{k}, \xi_{N}(s)\rangle\Big| > \delta \right)=0,
$$
then there exists a subsequence $\{Q_{N_{k}}^{H}\}_{k=0}^{\infty}$ and a probability measure $Q^{H}$ on $D([0, T], \mathcal{M})$
such that $Q_{N_{k}}^{H}$ weakly converges to $Q^{H}$ as $k \to \infty$.
\end{pro}

The next proposition claims that each subsequential limit $Q^{H}$ in Proposition\ref{relative compactness} is absolutely continuous with respect to Lebesgue measure on the torus for each time $t$, and its density has the value in $[0, 1]$ a.e.
The proof is the same as in \cite{KL} Section 4, pp.57, so we omit the proof.
 
\begin{pro}\label{absolute continuity}
All limit points $Q^{H}$ of $\{Q_{N}^{H}\}_{N}$ are concentrated on trajectories of absolutely continuous measures with respect to the Lebesgue measure for each time $t$, i.e., there exists a Borel set $W \subset D([0, T], \mathcal{M})$ such that $Q^{H}(W)=1$ and for every $\xi_{\cdot} \in W$ and every $t \in [0, T]$, $\xi_{t}$ is absolutely continuous with respect to the Lebesgue measure $du$. 
Moreover, the density $\rho(t, u):=d\xi_{t}/du$ satisfies that $0 \le \rho(t, u) \le 1$, $du$-a.e.
\end{pro}

To simplify the notation, we put $J(t, x):=J(t, \Phi_{N}(x)), H(t,x):=H(t, \Phi_{N}(x))$ for $J, H \in C^{1, 2}([0, T], \mathbb{T}^{d})$, respectively.

Then we have:

\begin{align*}
&N^{2}L_{N}^{H}\langle J_{t}, \xi_{N}(t) \rangle
=\frac{N^{2}}{2|V_{N}|}\sum_{e \in E_{N}}\exp\left[H(t, te)-H(t, oe)\right] \cdot (-\eta_{oe}\eta_{te}+\eta_{oe})(J(t, te)- J(t, oe))   \\
&=\frac{N^{2}}{2|V_{N}|}\sum_{x \in V_{N}}\sum_{e \in E_{N, x}}\left[\left\{\exp\left(H(t, te) - H(t, oe)\right)-1\right\}\cdot \left(J(t, te) - J(t, oe)\right)\eta_{oe}+\left(J(t, te)- J(t,  oe)\right)\eta_{oe}\right] \\
&-\frac{N^{2}}{4|V_{N}|}\sum_{e \in E_{N}}\left\{\exp\left(H(t, te) - H(t, oe)\right)- \exp\left(H(t, oe)-H(t, te)\right)\right\}\cdot \eta_{oe}\eta_{te}(J(t, te)- J(t, oe)).
\end{align*}

For $e \in E$, we denote the directional derivative along ${\bf v}(e)$ by
$$
\nabla_{{\bf v}(e)}H(t, x):=\sum_{i=1}^{d}\frac{\partial H}{\partial x_{i}}(t, x)v_{i}(e), \ \ \ \ \text{for $[0, T]\times \mathbb{T}^{d}$}.
$$

Applying the inequality $|e^{z} - 1 - z| \le (1/2)|z|^{2}e^{|z|}$ for $z \in \mathbb{R}$,
by the regularity of $H$ and by the compactness $\mathbb{T}^{d}$,
there exists a constant $C>0$ not depending on each point of $[0, T]\times \mathbb{T}^{d}$ such that for every $N$ and for every $e \in E_{N}$,
$$
\left| N\{\exp\left(H(t, te)-H(t, oe)\right)-1\} - \nabla_{{\bf v}(e)}H(t, oe)\right| \le \frac{C}{N},
$$
$$
\left| N\{\exp\left(H(t, oe)-H(t, te)\right)-1\} + \nabla_{{\bf v}(e)}H(t, oe)\right| \le \frac{C}{N},
$$
$$
\left| N\left\{J(t, te) - J(t, oe) \right\} - \nabla_{{\bf v}(e)}J(t, oe) \right| \le \frac{C}{N}.
$$

By the convergence of the combinatorial Laplacian in Section\ref{scaling limit}, we have that
$$
N^{2}\sum_{e \in E_{N, x}}\left(J(t, te) - J(t, oe)\right) = \frac{1}{2}\sum_{e \in E_{N, x}}\sum_{i,j=1}^{d}\frac{\partial^{2}J}{\partial x_{i} \partial x_{j}}(t, x)v_{i}(e)v_{j}(e) + o_{N}.
$$
Hence,

\begin{align}\label{N2L}
N^{2}L_{N}^{H}\langle J_{t}, \xi_{N}(t) \rangle	
&=\frac{1}{2|V_{N}|}\sum_{e \in E_{N}}\nabla_{{\bf v}(e)}H(t, oe)\cdot \nabla_{{\bf v}(e)}J(t, oe)\eta_{oe} \nonumber \\
&+ \frac{1}{4|V_{N}|} \sum_{x \in V_{N}}\sum_{e \in E_{N, x}}\sum_{i,j=1}^{d}\frac{\partial^{2}J}{\partial x_{i} \partial x_{j}}(t, x)v_{i}(e)v_{j}(e) \eta_{oe} \nonumber \\
&- \frac{1}{2|V_{N}|}\sum_{e \in E_{N}}\nabla_{{\bf v}(e)}H(t, oe)\cdot \nabla_{{\bf v}(e)}J(t, oe)\eta_{oe}\eta_{te} + o_{N}. 
\end{align}

We replace $\mathbb{P}_{N}^{H}$ by $Q_{N}^{H}$ regarding the empirical density $\xi_{N}$ as the measure on $D([0, T],\mathcal{M})$.
Let us prove the following lemma:

\begin{lem}\label{equi-continuity}
For every $J \in C(\mathbb{T}^{d})$ and for every $\delta >0$,
$$
\lim_{\gamma \to 0}\limsup_{N \to \infty}Q_{N}^{H}\left[\sup_{|t-s|\le \gamma}\Big|\langle J, \xi_{N}(t)\rangle - \langle J, \xi_{N}(s)\rangle \Big|>\delta\right]=0.
$$
\end{lem}

\begin{proof}
For every continuous functions $J: \mathbb{T}^{d} \to \mathbb{R}$, it holds that
$$
\langle J, \xi_{N}(t)\rangle - \langle J, \xi_{N}(s)\rangle=\int_{s}^{t}b_{N}\left(\eta^{N}(s)\right)ds + M_{N}(t) - M_{N}(s).
$$
Since by (\ref{N2L}) there exists a constant $C$ such that for large enough $N$, $b_{N}(t)=\langle (\partial/\partial t)J_{t}, \xi_{N}(t)\rangle + N^{2}L_{N}^{H}\langle J_{t}, \xi_{N}(t)\rangle \le C$ uniformly, we obtain that by the Chebychev inequality and by the triangular inequality, for every $\delta >0$, for every $\gamma >0$ and for large enough $N$, 
$$
Q_{N}^{H}\left[\sup_{|t-s|\le \gamma}\Big|\langle J, \xi_{N}(t)\rangle - \langle J, \xi_{N}(s)\rangle \Big| > \delta \right] \le (1/\delta)\mathbb{E}_{N}^{H}\left[C \gamma + 2 \sup_{0 \le t \le T} \left|M_{N}(t)\right|\right].
$$
Then by Lemma\ref{Doob},
$$
\limsup_{N \to \infty}Q_{N}^{H}\left[\sup_{|t-s| \le \gamma}\Big|\langle J, \xi_{N}(t)\rangle - \langle J, \xi_{N}(s)\rangle \Big| > \delta \right]\le C\frac{\gamma}{\delta}.
$$
Therefore for every $\delta >0$,
$$
\lim_{\gamma \to 0}\limsup_{N \to \infty}Q_{N}^{H}\left[\sup_{|t-s|\le \gamma}\Big|\langle J, \xi_{N}(t)\rangle - \langle J, \xi_{N}(s)\rangle \Big|>\delta \right]=0.
$$
It completes the proof.
\end{proof}

\subsection{Application of replacement theorem}\label{application}

We have the following estimate: There exists a constant $C(H, J)>0$ depending only on $H$ and $J$ such that for every $e \in E_{N}$ and for every $\underline \sigma, \underline \sigma' \in \Gamma_{N}$,
\begin{align*}
&\left|\nabla_{{\bf v}(\underline \sigma' e)}H(t, o\underline \sigma' e)\cdot \nabla_{{\bf v}(\underline \sigma' e)}J(t, o\underline \sigma' e) -\nabla_{{\bf v}(\underline \sigma e)}H(t, o\underline \sigma e)\cdot \nabla_{{\bf v}(\underline \sigma e)}J(t, o\underline \sigma e)\right|  \\
&\le C(H, J)\|\underline \sigma' - \underline \sigma \|_{1}\frac{1}{N}. \\
\end{align*}

Thus, by the regularity of $H$ and $J$, putting $G_{N}(o e):=\nabla_{{\bf v}( e)}H(t, o e)\cdot \nabla_{{\bf v}(e)}J(t, o e)$ for $e \in E_{N}$,
$$
\frac{1}{|V_{N}|}\sum_{e \in E^{0}}\sum_{\underline \sigma \in \Gamma_{N}}\left|G_{N}(o\underline \sigma e)- \widetilde{G_{N}}_{o\underline \sigma e,K}\right| \le C\frac{K}{N} \ \ \ \ \ \text{for some constant $C>0$}.
$$
Here we regard $G(\cdot)$ as a local function bundle independent of states and denote by $\widetilde{G_{N}}_{o\underline \sigma e,K}$ the local average. 
By the uniform continuity of the twice derivative of $J$, putting
$$
F_{N}(x):=\sum_{e \in E_{N, x}}\sum_{i,j=1}^{d}\frac{\partial^{2}J}{\partial x_{i} \partial x_{j}}(t, x)v_{i}(e)v_{j}(e) \ \ \ \ \text{for $x \in V_{N}$},
$$
it holds that
$$
\frac{1}{|V_{N}|}\sum_{x \in  D_{x_{0}}}\sum_{\underline \sigma \in \Gamma_{N}}\left| F_{N}(\underline \sigma x) - \widetilde{F_{N}}_{\underline \sigma x, K}\right| =o_{N}.
$$
Here we also regard $F_{N}(\cdot)$ as a local function bundle.

By the above argument, we obtain that

\begin{align*}
N^{2}L_{N}^{H}\langle J_{t}, \xi_{N}(t) \rangle	
&=\frac{1}{2|V_{N}|}\sum_{e \in E^{0}}\sum_{\underline \sigma \in \Gamma_{N}} G_{N}(o\underline \sigma e)\cdot \widetilde{\eta}_{o\underline \sigma e, K} 
+ \frac{1}{4|V_{N}|}\sum_{x \in  D_{x_{0}}}\sum_{\underline \sigma \in \Gamma_{N}}F_{N}(\underline \sigma x)\cdot \widetilde{\eta}_{\underline \sigma x, K} \\
&- \frac{1}{2|V_{N}|}\sum_{e \in E^{0}}\sum_{\underline \sigma \in \Gamma_{N}} G_{N}(o\underline \sigma e) \widetilde{f^{(e)}}_{o\underline \sigma e, K} + o_{N}.
\end{align*}

Here $f^{(e)}$ is the local function bundle appearing in the third example in Section\ref{local function bundles} and $\widetilde{f^{(e)}}_{\cdot, K}$ its local average.

By applying Theorem\ref{super exponential estimate} and by the continuity of $(\partial^{2}/\partial x_{i}\partial x_{j}) J$, $\nabla_{{\bf v}(e)}H$ and $\nabla_{{\bf v}(e)}J$ on the compact space $[0, T] \times \mathbb{T}^{d}$, 
it holds that for every $t \in [0, T]$ and for every $\delta >0$,

\begin{align*}
\lim_{K \to \infty}\limsup_{\epsilon \to 0}\limsup_{N \to \infty}\frac{1}{|\Gamma_{N}|}\log \mathbb{P}_{N}^{H}\Bigg(\int_{0}^{t}\Big|\frac{1}{|V_{N}|}\sum_{x \in D_{x_{0}}}\sum_{\underline \sigma \in \Gamma_{N}} F_{N}(\underline \sigma x)\left\{ \widetilde{\eta}_{\underline \sigma x, K} - \overline{\eta}_{\underline \sigma x, \epsilon N}\right\}\Big|ds > \delta \Bigg)= -\infty,
\end{align*}

\begin{align*}
\lim_{K \to \infty}\limsup_{\epsilon \to 0}\limsup_{N \to \infty}\frac{1}{|\Gamma_{N}|}\log \mathbb{P}_{N}^{H}\Bigg(\int_{0}^{t}\Big|\frac{1}{|V_{N}|}\sum_{e \in E^{0}}\sum_{\underline \sigma \in \Gamma_{N}}
G_{N}(o\underline \sigma e) \left\{\widetilde{\eta}_{o\underline \sigma e, K} - \overline{\eta}_{o\underline \sigma e, \epsilon N}\right\}\Big|ds > \delta \Bigg)= -\infty,
\end{align*}

and

\begin{align*}
\lim_{K \to \infty}\limsup_{\epsilon \to 0}\limsup_{N \to \infty}\frac{1}{|\Gamma_{N}|}\log \mathbb{P}_{N}^{H}\Bigg(\int_{0}^{t}\Big|\frac{1}{|V_{N}|}\sum_{e \in E^{0}}\sum_{\underline \sigma \in \Gamma_{N}}
G_{N}(o\underline \sigma e) \left\{\widetilde{f^{(e)}}_{o\underline \sigma e, K} - \left(\overline{\eta}_{o\underline \sigma e, \epsilon N}\right)^{2}\right\}\Big|ds > \delta \Bigg)= -\infty.
\end{align*}
Here we use $\langle f^{(e)}_{x} \rangle(\rho)=\rho^{2}$ for every $x \in V_{N}$ in the third estimate above.
By the triangular inequality, it holds that for every $t \in [0, T]$ and for every $\delta >0$,

\begin{align*}
&\limsup_{\epsilon \to 0}\limsup_{N \to \infty} \mathbb{P}_{N}^{H}\Bigg(\int_{0}^{t}\Big|N^{2}L_{N}^{H}\langle J_{s}, \xi_{N}(s) \rangle \\
&- \frac{1}{2|V_{N}|}\sum_{e \in E^{0}}\sum_{\underline \sigma \in \Gamma_{N}}\nabla_{{\bf v}(e)}H(s, o\underline \sigma e)\cdot \nabla_{{\bf v}(e)}J(s, o\underline \sigma e)\cdot( \overline{\eta}_{\underline \sigma x_{0}, \epsilon N}) \\
&- \frac{1}{4|V_{N}|}\sum_{x \in  D_{x_{0}}}\sum_{\underline \sigma \in \Gamma_{N}}\sum_{e \in E_{N, \underline \sigma x}}\sum_{i,j=1}^{d}\frac{\partial^{2}J}{\partial x_{i} \partial x_{j}}(t, \underline \sigma x)v_{i}(e)v_{j}(e)\cdot (\overline{\eta}_{\underline \sigma x_{0}, \epsilon N}) \\
&+ \frac{1}{2|V_{N}|}\sum_{e \in E^{0}}\sum_{\underline \sigma \in \Gamma_{N}}\nabla_{{\bf v}(e)}H(s, o\underline \sigma e)\cdot \nabla_{{\bf v}(e)}J(s, o\underline \sigma e)\cdot( \overline{\eta}_{\underline \sigma x_{0}, \epsilon N})^{2}
 \Big|ds > \delta \Bigg)= 0.
\end{align*}

Applying the convergence of the combinatorial Laplacian in Section \ref{scaling limit}, we have that
$$
\frac{1}{4|V_{N}|}\sum_{x \in  D_{x_{0}}}\sum_{\underline \sigma \in \Gamma_{N}}\sum_{e \in E_{N, \underline \sigma x}}\sum_{i,j=1}^{d}\frac{\partial^{2}J}{\partial x_{i} \partial x_{j}}(t, \underline \sigma x)v_{i}(e)v_{j}(e)
=\frac{1}{|\Gamma_{N}|}\sum_{\underline \sigma \in \Gamma_{N}} \nabla \mathbb{D} \nabla J(t, \underline \sigma x_{0}) + o_{N}.
$$

Recall that
$$
\langle J_{t}, \xi_{N}(t)\rangle - \langle J_{0}, \xi_{N}(0) \rangle =\int_{0}^{t}\left\{\langle \partial_{s} J_{s}, \xi_{N}(s) \rangle + N^{2}L_{N}^{H}\langle J_{s}, \xi_{N}(s)\rangle \right\}ds + M_{N}(t).
$$

By Lemma\ref{Doob} and by the Chebychev inequality, for every $\delta >0$, 
$$
Q_{N}^{H}\left(\sup_{0 \le t \le T}|M_{N}(t)| > \delta \right) \le (1/\delta)\mathbb{E}_{N}^{H}\left(\sup_{0 \le t \le T}|M_{N}(t)|\right) \to 0 \ \ \ \ \ \ \text{as $N \to \infty$}.
$$
Furthermore, by the triangular inequality, we have that for every $\delta > 0$ and for every $t \in [0, T]$,

\begin{align*}
&\limsup_{\epsilon \to 0}\limsup_{N \to \infty}Q_{N}^{H}\Bigg( \Bigg| \langle J_{t}, \xi_{N}(t)\rangle - \langle J_{0}, \xi_{N}(0) \rangle - \int_{0}^{t}\Bigg\{\langle \partial_{s} J_{s}, \xi_{N}(s) \rangle     \\
&+ \frac{1}{2|V_{N}|}\sum_{e \in E^{0}}\sum_{\underline \sigma \in \Gamma_{N}}\nabla_{{\bf v}(e)}H(t, o\underline \sigma e)\cdot \nabla_{{\bf v}(e)}J(t, o\underline \sigma e)\cdot( \overline{\eta}_{\underline \sigma x_{0}, \epsilon N})
+ \frac{1}{|\Gamma_{N}|}\sum_{\underline \sigma \in \Gamma_{N}}\nabla \mathbb{D} \nabla J(t, \underline \sigma x_{0})\cdot (\overline{\eta}_{\underline \sigma x_{0}, \epsilon N})	\\
&- \frac{1}{2|V_{N}|}\sum_{e \in E^{0}}\sum_{\underline \sigma \in \Gamma_{N}}\nabla_{{\bf v}(e)}H(t, o\underline \sigma e)\cdot \nabla_{{\bf v}(e)}J(t, o\underline \sigma e)\cdot \left(\overline \eta_{\underline \sigma x_{0}, \epsilon N}\right)^{2}\Bigg\}ds \Bigg| > \delta \Bigg)=0.  \\
\end{align*}

By Lemma\ref{characteristic function}, we replace $\overline \eta_{\underline \sigma x_{0}, \epsilon N}$ by $\langle \xi_{N}, \chi_{\Phi_{N}(\underline \sigma x_{0}), \epsilon}\rangle$ and the summation for $\underline \sigma \in \Gamma_{N}$ by the integral.
Since by Lemma\ref{equi-continuity} and by Proposition\ref{relative compactness}, the sequence $\{Q_{N}^{H}\}$ is relatively compact in the weak topology, for a limit point $Q^{H}$
there exists a subsequence $\{Q_{N_{k}}^{H}\}$ weakly converging to $Q^{H}$.
By Proposition\ref{absolute continuity}, 
the empirical density $\xi_{N_{k}}$ concentrates on an absolutely continuous trajectory $\rho_{t}d\mu$ as $k \to \infty$.
By the assumption of Theorem\ref{main}, we replace $\langle J_{0}, \xi_{N}(0)\rangle$ by $\langle J_{0}, \rho_{0} \rangle$, and then we have that for every $\delta >0$ and for every $t \in [0,T]$,

\begin{align*}
&\limsup_{\epsilon \to 0}Q^{H}\Bigg( \Bigg| \langle J_{t}, \rho_t \rangle - \langle J_{0}, \rho_0 \rangle - \int_{0}^{t}\Bigg\{\langle \partial_{s} J_{s}, \rho_{s} \rangle      \\
&+ \frac{1}{2|V_{0}|}\sum_{e \in E_{0}}\int_{\mathbb{T}^{d}}\nabla_{{\bf v}(e)}H(t, z)\cdot \nabla_{{\bf v}(e)}J(t, z)\cdot\langle \rho_{t}, \chi_{z, \epsilon}\rangle \mu(dz)
+ \int_{\mathbb{T}^{d}}\nabla \mathbb{D} \nabla J(t, z)\cdot \langle \rho_{t}, \chi_{z, \epsilon}\rangle \mu(dz) \\
&- \frac{1}{2|V_{0}|}\sum_{e \in E_{0}}\int_{\mathbb{T}^{d}}\nabla_{{\bf v}(e)}H(t, z)\cdot \nabla_{{\bf v}(e)}J(t, z)\cdot \langle \rho_{t}, \chi_{z, \epsilon}\rangle^{2}\mu(dz)\Bigg\}ds \Bigg| > \delta \Bigg)=0.
\end{align*}

Here we replace ${\bf v}(e)$ for $e \in E^{0}$ by ${\bf v}(e)$ for $e \in E_{0}$ by the $\Gamma$-invariance of ${\bf v}(\cdot)$.

By the Lebesgue dominated convergence theorem as $\epsilon \to 0$ and by the triangular inequality, we have that for every $\delta >0$ and for every $t \in [0,T]$,

\begin{align*}
&Q^{H}\Bigg( \Bigg| \langle J_{t}, \rho_t \rangle - \langle J_{0}, \rho_0 \rangle - \int_{0}^{t}\Bigg\{\langle \partial_{s} J_{s}, \rho_{s} \rangle      
+ \frac{1}{2|V_{0}|}\sum_{e \in E_{0}}\langle \nabla_{{\bf v}(e)}J \cdot \nabla_{{\bf v}(e)}H, \rho_{t}\rangle 
+ \langle \nabla \mathbb{D} \nabla J, \rho_{t} \rangle	\\
&\ \ \ \ \ \ \ \ \ \ \ \ \ \ \ \ \ \ \ \ \ \ \ \ \ \ \ \ \ \ \ \ \ \ \ \ \ \ \ \ \ \ \ \ \ \ \ \ \ \ \ \ \ \ \ \ \ \ \ \ \ \ \ \ \ \ \ \ \ \ \ \ - \frac{1}{2|V_{0}|}\sum_{e \in E_{0}}\langle  \nabla_{{\bf v}(e)}J \cdot \nabla_{{\bf v}(e)}H, \rho_{t}^{2}\rangle \Bigg\}ds \Bigg| > \delta \Bigg)=0.
\end{align*}

This shows $\xi_{N_{k}}(t)$ concentrates on $\rho$, which is a weak solution of the quasi-linear parabolic equation (\ref{pde1}).
Furthermore, $\rho$ has finite energy by Lemma\ref{energy}.
By the uniqueness result of the weak solution Lemma\ref{uniqueness} in Section\ref{appendixB}, we conclude that the limit point $Q^{H}$ of $\{Q_{N}^{H}\}_{N}$ is unique and $\xi_{N}$ concentrates on $\rho d\mu$ as $N$ goes to the infinity. 
That is, for every $\delta >0$,
$$
\lim_{N \to \infty}Q_{N}^{H}\Big(d_{Sk.}\left(\xi_{N}, \rho d\mu \right) > \delta \Big)=0,
$$
where $d_{Sk.}$ is the Skorohod distance in $D([0,T], \mathcal{M})$.
In particular, since $T>0$ is arbitrary, it follows that for every $t \ge 0$, for every $\delta >0$ and for every continuous functions $J \in C(\mathbb{T}^{d})$,
$$
\lim_{N \to \infty}\mathbb{P}_{N}^{H}\left[\left|\frac{1}{|V_{N}|}\sum_{x \in V_{N}}J\left(\Phi_{N}(x)\right)\eta^{N}_{x}(t)- \int_{\mathbb{T}^{d}}J(u)\rho(t,u)\mu(du) \right| > \delta \right]=0.
$$
It completes the proof of Theorem\ref{main}.

\section{Appendix;A}\label{appendixA}

\subsection{Approximation by combinatorial metrics}\label{approximation}

Take a $\mathbb{Z}$-basis $\sigma_{1}, \dots, \sigma_{d}$ of $\Gamma$ and identify $\Gamma$ with $\mathbb{Z}^{d}$.
Define the standard generator system of $\Gamma$ by $S=\{\sigma_{1}, \dots, \sigma_{d}, -\sigma_{1}, \dots, -\sigma_{d}\}$. We introduce the length function associated with $S$, $|\cdot| : \Gamma \to \mathbb{N}$ by
$$
|\sigma|:=\min\left\{l \ \Big| \ \sigma=\sum_{k=1}^{l}\epsilon_{i_{k}}\sigma_{i_{k}}, \epsilon_{i_{k}} \in \{-1, 1\}, i_{k} \in \{1, \dots, d\}\right \}
$$ 
for $\sigma \in \Gamma$. 
Then the map $(\sigma, \sigma') \in \Gamma \times \Gamma \mapsto |\sigma - \sigma'| \in \mathbb{N}$ induces the metric in $\Gamma$, which is called the word metric associated with $S$. 
Let $\underline{\sigma_{1}}, \dots, \underline{\sigma_{d}} \in \Gamma_{N}$ be the image of $\sigma_{1}, \dots, \sigma_{d}$ by the natural homomorphism $\Gamma \to \Gamma_{N}$. 
Then $\underline S:=\{\underline{\sigma_{1}}, \dots, \underline{\sigma_{d}}, -\underline{\sigma_{1}}, \dots, -\underline
{\sigma_{d}}\}$ generates $\Gamma_{N}$. 
The length function associated with $\underline S$, $|\cdot|: \Gamma_{N} \to \mathbb{N}$ is also defined in the same way.
To abuse the notation, we denote the word metric in $\Gamma_{N}$ associated with $\underline S$ by the same symbol $|\cdot|$.

We define an $l^{1}$-norm in $\Gamma \otimes \mathbb{R} \cong \mathbb{R}^{d}$ by 
$$
\|{\bf x}\|_{1}:=\sum_{i=1}^{d}|x_{i}|
$$
for ${\bf x}=(x_{1}, \dots, x_{d}) \in \mathbb{R}^{d}$ and the distance $\widetilde d_{1}$ in $\mathbb{R}^{d}$ by $\widetilde d_{1}({\bf x}, {\bf x'}):=\|{\bf x}-{\bf x'}\|_{1}$ for ${\bf x}, {\bf x'} \in \mathbb{R}^{d}$.
Denote by $d_{1}$ the induced metric in $\mathbb{T}^{d}$ from $\widetilde  d_{1}$.

Fix $x_{0} \in V$ and a fundamental domain $D_{x_{0}} \subset V$ such that $x_{0} \in D_{x_{0}}$ and $D_{x_{0}}$ is connected in the following sense:
For any $x, y \in D_{x_{0}}$ there exist a path $e_{1}, \dots, e_{l}$ in $E$ such that $oe_{1}=x, te_{l}=y$ and $oe_{1}, te_{1}, \dots, oe_{l}, te_{l}$ are all in $D_{x_{0}}$.
This kind of set $D_{x_{0}}$ always exists if we take a spanning tree in $X_{0}$ and its lift in $X$.
To abuse the notation, we denote by $x_{0} \in V_{N}, D_{x_{0}} \subset V_{N}$ the images of $x_{0}, D_{x_{0}}$ by the covering map, respectively.
We also fix a fundamental domain $E^{0} \subset E$ which is identified with $E_{0}$.
To abuse the notation, we denote by $E^{0} \subset E_{N}$ the image of $E^{0}$ by the covering map.

We define the map $[\cdot]:V \to \Gamma$ as follows:
For $x \in V$, there exists a unique element $\sigma \in \Gamma$ such that $x \in \sigma D_{x_{0}}$ since $\Gamma$ acts on $X$ freely. 
Define $[x]:=\sigma$. 
Since there exists a constant $C_{0}:=\max_{x \in V}\|\Phi([x]x_{0})-\Phi(x)\|_{1}$ by the $\Gamma$-perodicity, we have that 
$$
\|\Phi([x]x_{0})-\Phi([z]x_{0})\|_{1} -2C_{0} \le \|\Phi(x)-\Phi(z)\|_{1} \le \|\Phi([x]x_{0})-\Phi([z]x_{0})\|_{1} + 2C_{0}
$$
for any $x, z \in V$.
Furthermore, since $\|\Phi([x]x_{0})- \Phi([z]x_{0})\|_{1}=\left|[x]-[z]\right|$, we have that
$$
\left|[x]-[z]\right| -2C_{0} \le \|\Phi(x)-\Phi(z)\|_{1} \le \left|[x] -[z]\right| +2C_{0}.
$$

As in Section\ref{harmonic}, suppose that we have an injective homomorphism $\psi: \Gamma \to \mathbb{R}^{d}$ such that
$$
\psi(\Gamma)=\left\{\sum_{i=1}^{d}k_{i}u_{i} \ | \ \text{$k_{i}$ integers and $u_{i} \in \mathbb{R}^{d}$, $u_{1}, \dots, u_{d}$ are linearly independent}\right\}.
$$
Take a fundamental parallelotope
$P:=\left\{\sum_{i=1}^{d}a_{i}u_{i} \ | \ 0 \le a_{i} < 1 \right\} \subset \Gamma \otimes \mathbb{R}$.
In the similar way to the above,
we define the map
$[\cdot]:\Gamma \otimes \mathbb{R} \to \Gamma$ as follows:
For ${\bf x} \in \Gamma \otimes \mathbb{R}$, there exists a unique element $\sigma \in \Gamma$ such that ${\bf x} \in \sigma P$. Define $[{\bf x}]:=\sigma$.
Since there exists a constant $C_{1}:=\max_{{\bf x} \in \Gamma \otimes \mathbb{R}}\|\Phi([{\bf x}]x_{0})- {\bf x}\|_{1}$ by the $\Gamma$-periodicity, we have that 
$$
\|\Phi([{\bf x}]x_{0})-\Phi([{\bf z}]x_{0})\|_{1} -2C_{1} \le \|{\bf x} - {\bf z}\|_{1} \le \|\Phi([{\bf x}]x_{0})-\Phi([{\bf z}]x_{0})\|_{1} + 2C_{1}
$$
for any ${\bf x}, {\bf z} \in \Gamma \otimes \mathbb{R}$.
Furthermore, we have that 
$$
\left|[{\bf x}]-[{\bf z}]\right| -2C_{1} \le \|{\bf x}-{\bf z}\|_{1} \le \left|[{\bf x}] -[{\bf z}]\right| +2C_{1}.
$$

Let us define an $\epsilon$-ball in $\mathbb{T}^{d}$ of the center ${\bf z} \in \mathbb{T}^{d}$ by 
$B^{1}_{{\bf z}}(\epsilon):=\{{\bf x} \in \mathbb{T}^{d} \ | \ d_{1}({\bf x}, {\bf z}) \le \epsilon \}$. We show the following:

\begin{lem}\label{ball approximation}
There exists a constant $C(\epsilon)$ depending only on $\epsilon$ such that for any ${\bf z} \in \mathbb{T}^{d}$ and for any $N \ge 1$,
$$
\left| \frac{vol\left(B^{1}_{{\bf z}}(\epsilon)\right)}{vol\left(\mathbb{T}^{d}\right)} - \frac{\left|\cup_{|\sigma| \le \epsilon N}\sigma D_{x_{0}} \right|}{|V_{N}|}\right| \le \frac{C(\epsilon)}{N}.
$$
Here $vol(T)$ stands for the volume of a Borel set $T$ and $|U|$ the cardinality of a set $U$.
\end{lem}

\begin{proof}
For any ${\bf z} \in \mathbb{T}^{d}$, take a lift $\tilde {\bf z} \in \mathbb{R}^{d}$ then $\|\Phi([\tilde {\bf z}]x_{0})- \tilde {\bf z}\|_{1} \le C_{1}$.
For sufficiently small $\epsilon >0$, take a lift $\widetilde B^{1}_{\tilde{\bf z}}(\epsilon) \subset \mathbb{R}^{d}$ of $B^{1}_{{\bf z}}(\epsilon) \subset \mathbb{T}^{d}$. Again, from the above argument, it holds that $\cup_{|\sigma| \le \epsilon N -2C_{1}}\sigma [\tilde{\bf z}]P \subset \widetilde B^{1}_{\tilde{\bf z}}(N\epsilon) \subset \cup_{|\sigma| \le \epsilon N +2C_{1}}\sigma [\tilde{\bf z}]P$.

Note that $vol\left(\widetilde B^{1}_{\tilde{\bf z}}(\epsilon)\right)=vol\left(B^{1}_{{\bf z}}(\epsilon)\right)$, and thus
$\left|N^{d}vol\left(B^{1}_{{\bf z}}(\epsilon)\right)- vol\left( \cup_{|\sigma| \le \epsilon N}\sigma [\tilde{\bf z}]P\right)\right| \le vol\left(P \right)(2^{d}/d!)\left((\epsilon N +2C_{1})^{d} -(\epsilon N -2C_{1})^{d}\right)$.

Since $\left|\{\sigma \in \Gamma \ | \ |\sigma| \le \epsilon N\}\right|=vol\left( \cup_{|\sigma| \le \epsilon N}\sigma [\tilde{\bf z}]P\right) / vol\left(P\right)= \left| \cup_{|\sigma| \le \epsilon N}\sigma [\tilde{\bf z}]D_{x_{0}}\right| / |V_{0}|$ and $vol\left(P\right)=vol\left(\mathbb{T}^{d}\right)$,
it concludes that there exists a constant $C(\epsilon)>0$ depending only on $\epsilon$ such that
$$
\left| \frac{vol\left(B^{1}_{{\bf z}}(\epsilon)\right)}{vol\left(\mathbb{T}^{d}\right)} - \frac{\left|\cup_{|\sigma| \le \epsilon N}\sigma[\tilde {\bf z}]D_{x_{0}} \right|}{N^{d}|V_{0}|}\right| \le \frac{C(\epsilon)}{N}.
$$
The cardinality of the set $\cup_{|\sigma| \le \epsilon N}\sigma D_{x_{0}}$ is invariant under translation.
It completes the proof.
\end{proof}

Let us define a measure $\mu$ on $\mathbb{T}^{d}$ by $\mu:=\left(1/vol\left(\mathbb{T}^{d}\right)\right)d{\bf x}$.
Let $\chi_{{\bf z}, \epsilon}:\mathbb{T}^{d} \to \mathbb{R}$ be a characteristic function defined by
$$
\chi_{{\bf z}, \epsilon}:=\frac{1}{\mu\left(B^{1}_{{\bf z}}(\epsilon)\right)}1_{B^{1}_{{\bf z}}(\epsilon)} \ \ \ \ \ \text{on $\mathbb{T}^{d}$}.
$$
For the empirical density $\xi_{N}:=\left(1/|V_{N}|\right)\sum_{x \in V_{N}}\eta_{x}\delta_{\Phi_{N}(x)}$ on $\mathbb{T}^{d}$,
$\eta \in Z_{N}$, then we have the following lemma.

\begin{lem}\label{characteristic function}
There exists a constant $C(\epsilon) >0$ depending only on $\epsilon >0$, such that for any $\eta \in Z_{N}$ and any $z \in V_{N}$,
$$
\left| \langle \xi_{N}, \chi_{\Phi_{N}(z),\epsilon}\rangle -\frac{1}{\left|\cup_{|\sigma|\le \epsilon N}\sigma [z]D_{x_{0}}\right|}\sum_{x \in \cup_{|\sigma|\le \epsilon N}\sigma [z]D_{x_{0}}} \eta_{x}\right| 
\le 
\frac{C(\epsilon)}{N},
$$
where 
$$
\langle \xi_{N}, \chi_{\Phi_{N}(z),\epsilon}\rangle=\frac{1}{|V_{N}|}\sum_{x \in V_{N}}\chi_{\Phi_{N}(z),\epsilon}(\Phi_{N}(x))\eta_{x}.
$$
\end{lem}

\begin{proof}
Take a lift $\tilde z \in V$ of $z \in V_{N}$ and a lift $\widetilde B^{1}_{(1/N)\Phi(\tilde z)}(\epsilon) \subset \Gamma \otimes \mathbb{R}$ of $B^{1}_{\Phi_{N}(z)}(\epsilon) \subset \mathbb{T}^{d}$.
In the similar way to the proof of Lemma\ref{ball approximation}, we obtain that
$$
\bigcup_{|\sigma|\le \epsilon N -2C_{0}}\sigma [\tilde z]D_{x_{0}} 
\subset 
\left\{x \in V \ \Bigg| \ \left\|\frac{1}{N}\Phi(x) - \frac{1}{N}\Phi(\tilde z)\right\|_{1} \le \epsilon \right\}
\subset 
\bigcup_{|\sigma|\le \epsilon N +2C_{0}}\sigma [\tilde z]D_{x_{0}}.
$$
Furthermore, we take a lift $\widetilde \eta \in Z$ of $\eta \in Z_{N}$, then it holds that 
$\sum_{x \in V_{N}, \Phi_{N}(x) \in B^{1}_{\Phi_{N}(z)}(\epsilon)} \eta_{x}=\sum_{x \in V_{N}, (1/N)\Phi(x) \in\widetilde B^{1}_{(1/N)\Phi(\tilde z)}(\epsilon)} \widetilde \eta_{x}$ and 
$$
\left|\frac{1}{|V_{N}|}\sum_{x \in V, \left\|\frac{1}{N}\Phi(x)-\frac{1}{N}\Phi(\tilde z)\right\|_{1}\le \epsilon} \widetilde \eta_{x} - \frac{1}{|V_{N}|}\sum_{x \in \cup_{|\sigma|\le \epsilon N}\sigma [\tilde z]D_{x_{0}}} \widetilde \eta_{x}\right|
\le 
\frac{1}{|V_{N}|}\sum_{x \in \cup_{\epsilon N - 2C_{0} \le |\sigma|\le \epsilon N +2C_{0}}\sigma [\tilde z]D_{x_{0}}}\widetilde \eta_{x}.
$$
The last term is bounded by $(|V_{0}|/|V_{N}|)(2^{d}/d!)\left((\epsilon N +2C_{0})^{d}-(\epsilon N -2 C_{0})^{d}\right)$, and thus there exists a constant $C_{1}(\epsilon)$ depending only on $\epsilon$ such that
$$
\left|\frac{1}{|V_{N}|}\sum_{x \in V, \left\|\frac{1}{N}\Phi(x)-\frac{1}{N}\Phi(\tilde z)\right\|_{1}\le \epsilon} \widetilde \eta_{x} - \frac{1}{|V_{N}|}\sum_{x \in \cup_{|\sigma|\le \epsilon N}\sigma [\tilde z]D_{x_{0}}} \widetilde \eta_{x}\right|
\le 
\frac{C_{1}(\epsilon)}{N}.
$$
By Lemma\ref{ball approximation} and $\left|\cup_{|\sigma| \le \epsilon N}\sigma[(1/N)\Phi(z)]D_{x_{0}} \right|=\left|\cup_{|\sigma| \le \epsilon N}\sigma D_{x_{0}} \right|=\left|\cup_{|\sigma| \le \epsilon N}\sigma[z]D_{x_{0}} \right|$,
\begin{align*}
\left| \frac{1}{|V_{N}|\mu \left(B^{1}_{\Phi_{N}(z)}(\epsilon)\right)} - \frac{1}{\left|\cup_{|\sigma| \le \epsilon N}\sigma[z]D_{x_{0}} \right|}\right| 
&\le
\frac{1}{|V_{N}|\mu\left(B^{1}_{\Phi_{N}(z)}(\epsilon)\right)}\frac{1}{\left|\cup_{|\sigma| \le \epsilon N}\sigma[z]D_{x_{0}} \right|}\frac{C(\epsilon)}{N} |V_{N}|		\\
&\le 
\frac{C_{2}(\epsilon)}{N^{d+1}},
\end{align*}
where $C_{2}(\epsilon)$ is a constant depending only on $\epsilon$.
Finally,
\begin{align*}
\left|\langle \xi_{N}, \chi_{\Phi_{N}(z),\epsilon}\rangle -\frac{1}{\left|\cup_{|\sigma|\le \epsilon N}\sigma [z]D_{x_{0}}\right|}\sum_{x \in \cup_{|\sigma|\le \epsilon N}\sigma [z]D_{x_{0}}} \eta_{x}\right|
&\le \frac{1}{\mu\left(B^{1}_{\Phi_{N}(z)}(\epsilon)\right)}\frac{C_{1}(\epsilon)}{N} + \frac{C_{2}(\epsilon)}{N^{d+1}}|V_{0}|(\epsilon N)^{d} \\
&\le \frac{C_{3}(\epsilon)}{N},
\end{align*}
where $C_{3}(\epsilon)$ is a constant depending only on $\epsilon$. It completes the proof.
\end{proof}

\section{Appendix;B}\label{appendixB}

\subsection{Energy estimate}

In this section, we prove the following lemma.

\begin{lem}[Energy estimate]\label{energy}
Suppose that $\{Q_{N}\}_{N \ge 1}$ is a sequence of probability measures on $D([0,T],\mathcal{M})$.
For any limit point $Q^{*}$ of $\{Q_{N}\}_{N \ge 1}$, 
$Q^{*}$-a.s. there exists a measurable function $\rho(t,u)$ such that $\xi_{t}=\rho_{t}d\mu$, $\rho$ has 
$$
\frac{\partial}{\partial x_{i}}\rho \in L^{2}([0,T]\times \mathbb{T}^{d}),
$$
for $1 \le i \le d$,
and satisfies
$$
\int_{0}^{T}\int_{\mathbb{T}^{d}}\frac{\partial}{\partial x_{i}}J \rho d\mu dt= - \int_{0}^{T}\int_{\mathbb{T}^{d}}J\frac{\partial}{\partial x_{i}}\rho d\mu dt,
$$
for every $J \in C^{0,1}([0,T]\times \mathbb{T}^{d})$ and $1 \le i \le d$.
\end{lem}

\begin{proof}

For fixed $x_{0} \in V_{N}$, we define a lattice $X_{N}^{S}=(V_{N}^{S},E_{N}^{S})$ whose vertex set $V_{N}^{S}$ is the subset of $V_{N}$ in the following:
$V_{N}^{S}$ is the orbit of $x_{0}$ by $\Gamma_{N}$, i.e.,
$\Gamma_{N} x_{0} \subset V_{N}$.
Define $E_{N}^{S}$ the set of oriented edges $f$ such that $f=(x_{0}, \underline{\sigma}x_{0})$ for some $\underline{\sigma} \in \{\underline{\sigma}_{i},-\underline{\sigma}_{i} \ | \ 1 \le i \le d\}$.
Then $\Gamma_{N}$ acts on $X_{N}^{S}$ naturally.
A configuration $\eta$ on $X_{N}$ induces the one on $X_{N}^{S}$ by restriction.
We use the same symbol $\eta$ for this restriction.
For $J \in C^{1,2}([0,T]\times \mathbb{T}^{d})$ and for $i \in \{1,\dots,d\}$, we define
$\omega^{(i)}:E_{N}^{S} \to C^{1}([0,T],\mathbb{R})$,
$$
\omega^{(i)}_{f}:=
\begin{cases}
J(of) & \text{if there exists $\underline{\sigma} \in \Gamma_{N}$ such that $f=\underline{\sigma}(x_{0},\underline{\sigma_{i}}x_{0})$} \\
0	& \text{otherwise},
\end{cases}
$$
and 
$$
F_{J,N}^{(i)}(\eta):=\sum_{f \in E_{N}^{S}}\omega^{(i)}_{f}(\eta_{of}-\eta_{tf}).
$$
For any $\mu \in \mathcal{P}(Z_{N})$, we have that
\begin{align*}
\int_{Z_{N}}F_{J,N}^{(i)}d\mu
&= \int_{Z_{N}}\sum_{f \in E_{N}^{S}}\omega^{(i)}_{f}(\eta_{of}-\eta_{tf})\left(\sqrt{\frac{d\mu}{d\nu^{N}}}\right)^{2}(\eta)d\nu^{N} \\
&=- \int_{Z_{N}}\sum_{f \in E_{N}^{S}}\omega^{(i)}_{f}\eta_{of} \cdot \pi_{f}\left(\sqrt{\frac{d\mu}{d\nu^{N}}}\right)^{2}(\eta)d\nu^{N} \\
&=- \int_{Z_{N}}\sum_{f \in E_{N}^{S}}\omega^{(i)}_{f}\eta_{of} \cdot \left(\pi_{f} \sqrt{\frac{d\mu}{d\nu^{N}}}(\eta)\right)\cdot \left(\sqrt{\frac{d\mu}{d\nu^{N}}}(\eta^{f}) + \sqrt{\frac{d\mu}{d\nu^{N}}}(\eta)\right)d\nu^{N} \\
&=- \int_{Z_{N}}\sum_{f \in E_{N}^{S}}\omega^{(i)}_{f}\eta_{tf} \cdot \left(-\pi_{f} \sqrt{\frac{d\mu}{d\nu^{N}}}(\eta)\right)\cdot \sqrt{\frac{d\mu}{d\nu^{N}}}(\eta)d\nu^{N} \\
&\ \ \ \ \ \ \ \ \ \ \ \ \ \ \ \ \ \ \ \ \ \ \ \ \ \ \ \ \ \ \ \ - \int_{Z_{N}}\sum_{f \in E_{N}^{S}}\omega^{(i)}_{f}\eta_{of} \cdot \left(\pi_{f} \sqrt{\frac{d\mu}{d\nu^{N}}}(\eta)\right)\cdot \sqrt{\frac{d\mu}{d\nu^{N}}}(\eta)d\nu^{N} \\
&=- \int_{Z_{N}}\sum_{f \in E_{N}^{S}}\omega^{(i)}_{f}(\eta_{of}-\eta_{tf}) \cdot \left(\pi_{f} \sqrt{\frac{d\mu}{d\nu^{N}}}(\eta)\right)\cdot \sqrt{\frac{d\mu}{d\nu^{N}}}(\eta)d\nu^{N}  \\
&\le \sqrt{\int_{Z_{N}}\sum_{f \in E_{N}^{S}}(\omega_{f}^{(i)}(\eta_{of}-\eta_{tf}))^{2}\frac{d\mu}{d\nu^{N}}d\nu^{N}}\cdot\sqrt{\int_{Z_{N}}\sum_{f \in E_{N}^{S}}\left(\pi_{f}\sqrt{\frac{d\mu}{d\nu^{N}}}\right)^{2}d\nu^{N}}.
\end{align*}
We use the Cauchy-Schwarz inequality in the last inequality.
By using the argument in the proof of Lemma\ref{path}, there exists a constant $C$ such that
$$
\int_{Z_{N}}\sum_{f \in E_{N}^{S}}\left(\pi_{f}\sqrt{\frac{d\mu}{d\nu^{N}}}\right)^{2}d\nu^{N} \le C\int_{Z_{N}}\sum_{e \in E_{N}}\left(\pi_{e}\sqrt{\frac{d\mu}{d\nu^{N}}}\right)^{2}d\nu^{N}.
$$

Since 
$I_{N}(\mu)=\int_{Z_{N}}\sum_{e \in E_{N}}\left(\pi_{e}\sqrt{d\mu/d\nu^{N}}\right)^{2}d\nu^{N}$
and 
$\sum_{f \in E_{N}^{S}}(\omega^{(i)}_{f}(\eta_{of}-\eta_{tf}))^{2} \le 2\sum_{\underline{\sigma} \in \Gamma_{N}}J(\underline{\sigma}x_{0})^{2}$,
we get
$$
\int_{Z_{N}}F_{J,N}^{(i)}d\mu \le \sqrt{\sum_{\underline{\sigma} \in \Gamma_{N}}J(\underline{\sigma}x_{0})^{2}}\cdot \sqrt{16CI_{N}(\mu)}.
$$
Note that 
$\sum_{\underline{\sigma} \in \Gamma_{N}}J(\underline{\sigma}x_{0})^{2} \le 2 |\Gamma_{N}| \cdot \|J\|_{L^{2}(\mathbb{T}^{d})}^{2}$ 
for large enough $N$.
Consider for $a >0$ the self-adjoint operator 
$$
N^{2}L_{N} + a F_{N}^{(i)}: L^{2}(Z_{N},\nu^{N}) \to L^{2}(Z_{N},\nu^{N}),
$$
and suppose $\lambda_{N}^{(i)}(a)$ to be the largest eigenvalue of this operator.
By the variational formula
\begin{align*}
\lambda_{N}^{(i)}(a)&=\sup_{\mu \in \mathcal{P}(Z_{N})}\left\{a \int_{Z_{N}}F_{J,N}^{(i)}d\mu - N^{2}I_{N}(\mu)\right\} \\
&\le \sup_{\mu \in \mathcal{P}(Z_{N})}\left\{2a\sqrt{|\Gamma_{N}|\cdot \|J\|_{L^{2}(\mathbb{T}^{d})}^{2}}\cdot \sqrt{16CI_{N}(\mu)} - N^{2}I_{N}(\mu)\right\}.
\end{align*}

By the simple inequality for $p,q \ge 0$, $2pq - q^{2} \le p^{2}$,
the last formula is bounded by 
$a^{2}|\Gamma_{N}|\cdot \|J\|_{L^{2}(\mathbb{T}^{d})}^{2} \cdot (16C/N^{2})$.
We put $a=N$.
On the other hand,
$$
\frac{1}{|\Gamma_{N}|}NF_{J, N}^{(i)}=-\sum_{\underline{\sigma} \in \Gamma_{N}}\nabla_{u_{i}}J(o\underline{\sigma}(x_{0},\underline{\sigma_{i}}x_{0}))\cdot \eta_{o\underline{\sigma}(x_{0},\underline{\sigma_{i}}x_{0})} + o_{N},
$$
where $\nabla_{u_{i}}$ is the directional derivative along $u_{i}$ and $u_{i}=\psi(\sigma_{i}) \in \mathbb{R}^{d}$.

By the entropy inequality,
$$
\mathbb{E}_{N}\int_{0}^{T}NF_{J,N}^{(i)}dt \le \log \mathbb{E}_{N}^{eq}\exp\left\{N\int_{0}^{T}F_{J,N}^{(i)}dt\right\} +H(\mathbb{P}_{N} | \mathbb{P}_{N}^{eq}).
$$
Since $H(\mathbb{P}_{N} | \mathbb{P}_{N}^{eq}) \le |V_{N}|C'$ for some constant $C'$,
by the Feynman-Kac formula,
we obtain
$$
\limsup_{N \to \infty}\frac{1}{|\Gamma_{N}|}\mathbb{E}_{N}\int_{0}^{T}NF_{J,N}^{(i)}dt \le 16CT \|J\|_{L^{2}([0,T] \times \mathbb{T}^{d})}^{2} +|V_{0}|C'.
$$
For a limit point of $\{Q_{N}\}_{N\ge 1}$, $Q^{*}$, we get
$$
\mathbb{E}_{Q^{*}}\left[-\int_{0}^{T}\int_{\mathbb{T}^{d}}\nabla_{u_{i}}J\cdot \rho d\mu\right] \le 16CT \|J\|_{L^{2}([0,T] \times \mathbb{T}^{d})}^{2}+|V_{0}|C'.
$$
Denote a countable dense subset of $C^{0,1}([0,T]\times \mathbb{T}^{d})$ by $\mathcal{J}$,
we also get the following estimate:
$$
\mathbb{E}_{Q^{*}}\left[\sup_{J \in \mathcal{J}}\int_{0}^{T}\int_{\mathbb{T}^{d}}(-\nabla_{u_{i}}J)\rho d\mu dt -16CT\int_{0}^{T}\int_{\mathbb{T}^{d}}J^{2}d\mu dt \right] \le |V_{0}|C'.
$$
See $\cite{KL}$ pp.107, Section 5.7 for details.
Therefore for almost all $\rho$, there exists $B(\rho)$ such that for every $J \in C^{0,1}([0,T]\times \mathbb{T}^{d})$,
$$
\int_{0}^{T}\int_{\mathbb{T}^{d}}(-\nabla_{u_{i}}J)\rho d\mu dt -16CT\int_{0}^{T}\int_{\mathbb{T}^{d}}J^{2}d\mu dt \le B(\rho),
$$
that is,
$$
\Biggl| \int_{0}^{T}\int_{\mathbb{T}^{d}}(-\nabla_{u_{i}}J)\rho d\mu dt \Biggr| \le 2\sqrt{16CT\int_{\mathbb{T}^{d}}J^{2}d\mu dt}\cdot \sqrt{B(\rho)}.
$$
This implies the linear functional $l_{\rho}:C^{0,1}([0,T]\times \mathbb{T}^{d}) \to \mathbb{R}$ defined by
$l_{\rho}(J):=\int_{0}^{T}\int_{\mathbb{T}^{d}}(-\nabla_{u_{i}}J)\rho d\mu dt$ is extended on $L^{2}([0,T]\times\mathbb{T}^{d})$.
By the Riesz representation theorem, 
there exists $\nabla_{u_{i}}\rho \in L^{2}([0,T]\times\mathbb{T}^{d})$ such that
$$
\int_{0}^{T}\int_{\mathbb{T}^{d}}(-\nabla_{u_{i}}J)\rho d\mu dt=\int_{0}^{T}\int_{\mathbb{T}^{d}}J\nabla_{u_{i}}\rho d\mu dt
$$
for every $J \in C^{0,1}([0,T]\times \mathbb{T}^{d})$ and every $i=1, \dots, d$.
This yields Lemma\ref{energy}.
\end{proof}

\subsection{The uniqueness result}

We state the uniqueness result used in Section\ref{application}.
The following lemma follows from the argument by using the Grownwall inequality.

\begin{lem}\label{uniqueness}
For any $H \in C^{1,2}([0,T]\times \mathbb{T}^{d})$, a weak solution of the quasi-linear partial differential equation
$$
\frac{\partial}{\partial t}\rho =\nabla \mathbb{D} \nabla \rho -\frac{1}{2|V_{0}|}\sum_{e \in E_{0}}\nabla_{{\bf v}(e)}(\rho(1-\rho)\nabla_{{\bf v}(e)}H)
$$
with the measurable initial value $\rho_{0}:\mathbb{T}^{d} \to [0,1]$, of bounded energy, i.e, 
$$
\int_{0}^{T}\int_{\mathbb{T}^{d}}\|\nabla \rho\|^{2}d\mu dt < \infty
$$
is unique.
\end{lem}

\textbf{Acknowledgement.}

First, the author would like to thank Professor Tsuyoshi Kato for his constant encouragement and helpful suggestions.
He wishes to express his gratitude to Professors Motoko Kotani, Nobuaki Sugimine, Satoshi Ishiwata and Makiko Sasada for their valuable advice and helpful discussions. 
Second, he would like to thank Professor Yukio Nagahata for his helpful comments on this subject and the background of the hydrodynamic limit.
Third, he would like to thank Doctors Sei-ichiro Kusuoka and Makoto Nakashima for their advices on the presentation of this paper.
Fourth, he would like to thank Professor Kazumasa Kuwada for his valuable comments on the earlier version of the manuscript and his encouragement. 
The author could not write up this paper without a great deal of his advice.
Fifth,  the author partially carried out this work at RIKEN Center for Developmental Biology in Kobe.
He would like to thank Professors Hiroki R. Ueda and Yohei Koyama for their interest in this topic and their encouragement.
Sixth, the author partially carried out this work at Max Planck Institute for Mathematics in the Sciences in Leipzig.
He would like to thank Professors J\"{u}rgen Jost and Nihat Ay for helpful discussions.
The author is supported by the Research Fellowships of the Japan Society for the Promotion of Science for Young Scientists.

\end{document}